\documentclass[12pt]{article}

\usepackage{latexsym,amssymb,amsmath,amsthm}
\usepackage{algorithmic}
\usepackage{algorithm}
\usepackage{subfigure,graphicx}
\usepackage{setspace}
\usepackage{booktabs}
\usepackage{paralist}
\usepackage{xcolor}
\usepackage{framed}
\usepackage{fancybox}
\usepackage{hyperref}
\numberwithin{equation}{section}    

\usepackage{fullpage}

\usepackage{shortcuts}

\def\Pcal{\ensuremath{{\mathcal P}}}
\def\Mcal{\ensuremath{{\mathcal M}}}

\numberwithin{equation}{section}
\newtheorem{theorem}{Theorem}[section]
\newtheorem{lemma}{Lemma}[section]
\newtheorem{corollary}{Corollary}[section]
\newtheorem{proposition}{Proposition}[section]
\newtheorem{example}{Example}[section]
\newtheorem{definition}{Definition}[section]
\newtheorem{assumption}{Assumption}

\renewenvironment{proof}{\vspace{1ex}\noindent{\bf
 Proof.}\hspace{0.5em}}{\hfill\qed\vspace{3ex}}

\title{A Globally Linearly Convergent Method for Pointwise Quadratically Supportable Convex-Concave Saddle Point Problems}

\author{
D. Russell Luke\thanks{
Institut f\"ur Numerische und Angewandte Mathematik,
Universit\"at G\"ottingen,
37083 G\"ottingen, Germany.  The research of Luke was supported in part by the 
German Research Foundation grants SFB755-A4 and GRK2088-B5.
} and Ron Shefi\thanks{
Institut f\"ur Numerische und Angewandte Mathematik,
Universit\"at G\"ottingen,
37083 G\"ottingen, Germany. The research of Shefi was supported by the German Research Foundation grant SFB755-A4}
}

\date{\today}

\begin{document}
\maketitle

\begin{abstract}
We study the \emph{Proximal Alternating Predictor-Corrector} (PAPC) algorithm 
introduced recently by Drori, Sabach and Teboulle \cite{drori2015simple} to solve 
nonsmooth structured convex-concave
saddle point problems consisting of the sum of a smooth convex function, a finite
collection of nonsmooth convex functions and bilinear terms.
We introduce the notion of pointwise quadratic supportability, which is 
a relaxation of a standard strong convexity assumption and allows us to show that 
the primal sequence is R-linearly convergent to an optimal solution and the 
primal-dual sequence is globally Q-linearly convergent.
We illustrate the proposed method on total variation denoising
problems and on locally adaptive estimation in signal/image deconvolution and denoising with
multiresolution statistical constraints.
\end{abstract}

{\small \noindent {\bfseries 2010 Mathematics Subject
Classification:} {Primary 49J52, 49M20, 90C26;
Secondary 15A29, 47H09, 65K05, 65K10, 94A08. 
}

}

\section{Introduction}

We revisit the primal-dual first-order Proximal Alternating Predictor-Corrector (PAPC) splitting scheme 
introduced in \cite{drori2015simple} for 
solving a class of structured convex-concave saddle point problems involving a smooth
function, a sum of finitely many bilinear terms and finitely many nonsmooth functions. 
This model covers a wide array of applications in 
signal/image processing and machine learning, see for instance \cite{sra2011optimization,Schertzer15} 
and references therein.  

This paper makes three contributions, two theoretical and one practical.  
In the first of the theoretical aspects, we introduce in Section \ref{s:basics}
the notion of pointwise quadratic supportability, Definition 
\ref{d:psc},  that
is weaker than the typical assumption of strong convexity and allows one to treat common functions that,
while not strongly convex, nevertheless possess desirable properties from a variational point of view.  
Pointwise quadratic supportability can be shown to be implied by pointwise strong monotonicity of the 
gradients (or more generally subdifferentials), which reveals a fundamental connection between this 
property and functions whose subdifferentials are {\em submonotone} as defined and studied in 
in \cite{LukNguTam16}.  A more detailed study of these connections, however, is not our focus here. 

We are rather more narrowly focused on the second and third contributions of this paper, namely theoretical 
guarantees of global linear convergence of the PAPC algorithm studied in Section \ref{s:convergence}
and efficient implementations detailed in Section \ref{s:numerics} for 
the problem of simultaneous denoising and deconvolution with multiresolution statistical constraints.
Theorem \ref{q-rate} establishes global Q-linear convergence of the primal-dual iterates 
of the PAPC algorithm for convex problems with pointwise quadratically supportable objective functions
and full rank linear mappings.  A remarkable corollary of this result, Corollary \ref{t:uniqueness}, 
is that saddle point problems with this structure have unique solutions. 
In Corollary \ref{r-corollary}, we show global R-linear convergence of the primal sequence under the 
same assumptions.  A sublinear iteration complexity in terms of the saddle point gap function values 
was established in \cite{drori2015simple} for the ergodic sequence.  In our development, 
we show that global linear convergence of the {\em iterates} can be guaranteed under 
the assumption of pointwise quadratic supportability and full rank linear mappings.
Based on these results, one can  use stopping criteria that monitor only successive steps and provide 
an a posteriori error estimate on the distance to the set of exact solutions to the underlying problem.  

We illustrate these results in Section \ref{s:numerics} 
with two main applications, one conventional 
and one that is new research.  The conventional example is of smoothed total variation (TV) denoising.
The main motivation for us, however, is statistical multiscale image
denoising/deconvolution following \cite{frick2012statistical,frick2013statistical} for fluorescence
microscopic images (see also \cite{aspelmeier2015modern} for a review of fluorescence microscopy
techniques and statistical methods for them).  
The PAPC method is a full splitting approach where the gradient and the linear operators
involved are called explicitly without inversion, while the simple nonsmooth functions are evaluated 
individually via their proximity operators.   This allows for efficient implementations.  
The only other implementation we are aware of with similar convergence guarantees was reported in 
\cite{aspelmeier2016local}.  The numerical approach studied here can achieve results on the order of
minutes on a small ($32\times 32$) image or less than an hour on a full $1024\times 1024$ image. 
The  approach presented in \cite{aspelmeier2016local} required days for a $32\times 32$ image and 
processing on the full-scale images was not feasible.  
As with \cite{aspelmeier2016local}, the convergence analysis presented here allows one to 
establish error bounds on solutions to the model problem which, in turn provide
statistical guarantees on the numerical reconstruction \cite{frick2012statistical}.  
This is demonstrated for image denoising/deconvolution of stimulated emission depletion 
(STED) images \cite{hell1994breaking,klar2000fluorescence}.  An additional advantage of 
the PAPC algorithm and the analysis presented here is that a wider variety of regularizing 
objective functions can be incorporated easily.  

Some well-known primal-dual decomposition methods  are the Chambolle-Pock algorithm
\cite{chambolle2010primaldual} and Alternating Direction of Multipliers (ADM).
The method proposed by Chambolle-Pock -- shown in \cite{shefi2014rate} 
to be a equivalent to ADM with a weighted norm --
requires the computation of proximal steps in each of the primal 
and the dual spaces.  
The algorithm has been shown to be linearly convergent 
when both functions are uniformly convex, but this assumption is far too stringent.  
Recently in \cite{deng2016global}, global linear convergence was shown
under the assumptions of strict convexity and Lipschitz gradient on one of the two functions along
with certain rank assumptions on the linear mapping. However, ADM-like methods do not easily extend 
to problems involving the sum of finitely many composite nonsmooth terms, see e.g., \cite{chen2016direct} 
where even the convergence becomes an issue. 
In \cite{aspelmeier2016local} it was shown that the dual sequence generated by ADM -- or equivalently
the primal sequence of the Douglas-Rachford algorithm -- must eventually achieve a linear rate of convergence 
from any starting point for piecewise linear-quadratic convex augmented Lagrangians whenever the saddle 
points are isolated.

\subsection{Notation.}
Our setting is the real vector space $\Rn$ with the norm generated from the inner product.
The closed unit ball centered on the point $y\in\Rn$ is denoted by $\mathbb{B}(y)$.
The domain of an extended real-valued function $\varphi:\Rn\to \R\cup\{+\infty\}$ is $\dom(\varphi)\equiv \{z \in\Rn :
\varphi(z)<+\infty\}$.
The Fenchel conjugate of $\varphi$ is denoted by $\varphi^*$ and is defined by
$\varphi^*(u)=\sup_{z\in\Rn}\{\bform{z,u}-\varphi(z)\}$.  
The set of symmetric 
$n\times n$ positive (semi)-definite matrices is denoted by $\mathbb{S}^n_{++}$
($\mathbb{S}^n_+$). We use $M \succ 0$ $(M \succeq 0)$ to denote a positive (semi)definite
matrix.
For any $z \in \Rn$ and any $M \in \mathbb{S}^n_+$, we denote the semi-norm $\|z\|^2_M :=
\bform{z,Mz}$.
The operator norm is defined by $\|M\| = \max_{u\in \Rn} \{ \|Mu\| : \|u\|=1\}$ and
coincides with the spectral radius of $M$ whenever $M$ is symmetric.
If $A \neq 0$, $\sigma_{min} (A)$ denotes its smallest nonzero
singular value. Let $\{\zk\}_{k\in\N}$ be a sequence that converges to $z^*$. We say the convergence is
$Q$-linear if there exists $c \in (0, 1)$ such that
$\frac{\norm{\zkn-z^*}}{\norm{\zk-z^*}} \leq c\;$ for all $k$;  convergence is
$R$-linear if there exists a sequence $\etak$ such that $\|\zk-z^*\|\leq \etak$ and $\etak \rarr 0$\; $Q$-linearly
\cite[Chapter 9]{OrtegaRheinboldt70}.

We limit our discussion to proper (nowhere equal to $-\infty$ and finite at some point), 
lower semi-continuous (lsc), extended-valued (can take the value $+\infty$) functions.  
We will, in fact, limit our discussion to {\em convex} functions, but for our introduction of 
{\em pointwise quadratic supportability} in Definition \ref{d:psc} we formulate
this with as much generality as possible to emphasize that convexity is not central to this key feature.  
By the {\em subdifferential} of a function $\varphi$, denoted $\partial \varphi$, we mean the collection of all 
{\em subgradients} that can be written as limits of sequences of  {\em Fr\'echet subgradients} at nearby points; 
a vector $v$ is  a {\em (Fr\'echet) subgradient} of $\varphi$ at $y$, written $v\in\widehat{\partial} \varphi(y)$,
if
\begin{equation}\label{e:rsd}
 \liminf_{x\to y,~x\neq y}\frac{\varphi(x)- \varphi(y) - \langle{v}, {x-y}\rangle}{\|x-y\|}\geq 0.
\end{equation}
The functions of interest for us are {\em subdifferentially regular} on their domains, that is,  
the epigraphs of the functions are {\em Clarke regular} at points where they are finite \cite[Definition 7.25]{VA}.  
For our purposes it suffices to note that, for a function $\varphi$ that is subdifferentially regular at a 
point $y$, the subdifferential is nonempty and all subgradients are Fr\'echet subgradients, that is,
$\partial \varphi(y)=\widehat{\partial} \varphi(y)\neq\emptyset$.  Convex functions, in particular, are subdifferentially regular 
on their domains and the subdifferential has the particularly simple representation as the set of 
all vectors $v$ where 
\begin{equation}\label{e:csd}
\varphi(x)- \varphi(y) - \langle{v}, {x-y}\rangle \geq 0\quad\forall x.
\end{equation}  
For $\varphi :\Rn \rarr(-\infty,\infty]$ a proper, lsc and convex function and 
for any $u \in \Rn$ and $M \in \mathbb{S}^n_{++}$,
the proximal map associated with $\varphi$ with respect to the weighted Euclidean norm is uniquely defined by:
$$ \prox_M^\varphi(u) = \argmin_z \{ \varphi(z) + \sfrac{2}\normsq{z-u}_M : z\in\Rn \}. $$
When $M = \inv{c}I_n, c > 0$, we simply use the notation $\prox_c^\varphi(u)$.
We also recall the fundamental Moreau proximal identity \cite{moreau1965proximite}, that is, for any
$z \in\Rn$
\begin{equation}\label{e:Moreau}
 z = \prox_{M}^\varphi (z) + M \prox_{\inv{M}}^{\varphi^*}(\inv{M}(z)),
\end{equation}
where $\inv{M}$ is the inverse of $M\in \mathbb{S}^n_{++}$.

\section{The saddle point model and the algorithm} 
\label{s:basics}
This note focuses on the following \emph{primal} problem:
\begin{equation*}\tag{\Pcal}\label{ch:main}
p_* = \min_x \left\{ p(x):= f(x) + \sumi[p]g_i(A_i^T x) : x\in\Rn \right\}.
\end{equation*}

The following blanket assumptions on the problem's data hold throughout:
\begin{center}
\fbox{%
        \addtolength{\linewidth}{\fboxsep}%
        \addtolength{\linewidth}{\fboxrule}%
        \begin{minipage}{\linewidth}%
        \begin{assumption}\label{hyp:A1}$~$
\begin{enumerate}[(i)]
   \item\label{hyp:A1i}
   The function $f:\Rn\rightarrow \R$  
  is convex and  continuously
  differentiable with Lipschitz continuous gradient $\nabla f$ (constant $L_f$), that is for
  all $x, x'\in\Rn$, we have
  \begin{equation}\label{Lf}
  \norm{\nabla f(x) - \nabla f(x')}\leq L_f \norm{x-x'}.
  \end{equation}
  \item\label{hyp:A1ii} $g_i: \Rmi\rarr(-\infty, +\infty], i=1, \ldots, p$ is proper, lsc, and convex.
  \item\label{hyp:A1iii} The linear mappings $A_i : \Rmi \rarr \Rn$, $i=1, \ldots, p$ are full rank,
  that is, $\sigma^2_{min}(A_i) = \lambda_{min}(A_i^TA_i) > 0$.
    \item\label{hyp:A1iv} The set of optimal solutions for problem \eqref{ch:main}, denoted $X^*$, is 
    nonempty. 
  \end{enumerate}
\end{assumption}
 \end{minipage}%
}
\end{center}
\bigskip

Assumption \eqref{hyp:A1ii} implies that the function defined by $g(y) := \sumi[p] g_i(y_i)$
with $y = (y_1, \ldots, y_p) \in \Rm$ for $m =\sumi[p]m_i$ is proper, lsc and convex.
Assumption \eqref{hyp:A1iii} implies that the linear map  $\A :\Rm \rarr \Rn$ 
  given by $\A y = \sumi[p]A_i y_i$ is full rank.  Assumption \eqref{hyp:A1iv} 
  implies that the optimal value of the primal problem is finite. 
  
The assumption of Lipschitz continuous gradients \eqref{hyp:A1i}, while standard in the literature involving 
complexity bounds, is very strong indeed.  A more satisfying theory would involve only functions 
with Lipschitz continuous gradients on 
bounded domains.  But this would involve a reinvention of much of the 
theory.  The main issue is boundedness of the iterates, which is not a precursor to the 
convergence results of \cite{drori2015simple} upon which we build.   
Four our purposes, Lipschitz continuity suffices, and leaves a short path 
to the main result.  We leave the stronger results as an open challenge. 

We reformulate the problem \eqref{ch:main} as a convex-concave saddle point
problem, and then apply the primal-dual algorithm in \cite{drori2015simple} to find
a saddle point solution. 
The \emph{primal-dual} problem associated to \eqref{ch:main} consists of finding a saddle point of the
Lagrangian:
\begin{equation*}\tag{\Mcal}\label{p:sp}
\min_{x\in\Rn} \max_{y\in \Rm}\left\{K(x,y) := f(x) + \bform{x,\A y} - g^*(y) \right\}.
\end{equation*}

Assumption \ref{hyp:A1}\eqref{hyp:A1iv}  guarantees that the convex-concave function $K(\cdot,\cdot)$
has a saddle point, that is, there exists $(\hx,\hy) \in \Rn \times \Rm$ such that 
$$K(\hx, y) \leq K(\hx, \hy) \leq K(x, \hy) \qquad \forall x\in\Rn , y \in \Rm. $$
The existence of a saddle point corresponds to zero duality gap for the induced optimization
problems
\[
 p(x) = \sup_{y} \{K (x, y) : y\in\Rm \} \qquad q (y) = \inf_{x} \{ K (x, y) : x\in\Rn  \}.
\]
By weak duality, we have $\inf_{x\in\Rn} p (x) \geq \sup_{y\in \Rm} q (y)$.
In addition, under standard constraint qualifications
(e.g., \cite[Chapter 5]{aus-teb-book} or \cite[Theorem 2.3.4]{CUP}), $(\hx,\hy)$ is a saddle point of $K$ 
if and only if $\hx$ is an optimal solution of the
primal problem \eqref{ch:main}, $\hy$ is an optimal solution of the dual problem to \eqref{ch:main}.

Constraints are included in the model through the extended-valued functions $g_i$, which 
need not be smooth.  The properties of the function $f$ are crucial to the success of the 
algorithm.  In addition to Assumption \ref{hyp:A1} we will assume that $f$ is 
{\em pointwise quadratically supportable}. 

\begin{definition}[pointwise quadratically supportable mappings]\label{d:psc} $~$ 
\begin{enumerate}[(i)]
 \item\label{d:pscoer} A proper, extended-valued function  $\varphi : \Rn \to  \R\cup \{+\infty\}$ 
is said to be {\em pointwise quadratically supportable at $y$} if it is subdifferentially regular there and 
there exists a  neighborhood $V$ of $y$ and a constant $\mu>0$ such that   
 \begin{equation}\label{e:psc}
 (\forall v\in\partial \varphi(y))\quad \varphi(x)\geq \varphi(y) + \bform{v, x-y}+ \frac{\mu}{2} \normsq{x-y}, \quad
 \forall  x\in V.
\end{equation}
If for each bounded neighborhood $V$ of $y$ there exists a constant $\mu>0$ such that 
\eqref{e:psc} holds, then the function $\varphi$ is said to be {\em
pointwise quadratically supportable at $y$ on bounded sets}.
If \eqref{e:psc} holds with one and the same constant $\mu>0$ on all neighborhoods $V$, then $\varphi$
 is said to be {\em uniformly} pointwise quadratically supportable
at $y$.  
\item\label{d:scoer} 
A proper, extended-valued function  $\varphi : \Rn \to  \R\cup \{+\infty\}$ 
is said to be {\em strongly coercive at $y$} if it is subdifferentially regular on $V$ and 
there exists a  neighborhood $V$ of $y$ and a constant $\mu>0$ such that   
 \begin{equation}\label{e:scoer}
 (\forall v\in\partial \varphi(z))\quad\varphi(x)\geq \varphi(z) + \bform{v, x-z}+ \frac{\mu}{2} \norm{x-z}^2, \quad
 \forall  x, z \in V.
\end{equation}
If for each bounded neighborhood $V$ of $y$ there exists a constant $\mu>0$ such that 
\eqref{e:scoer} holds, then the function $\varphi$ is said to be {\em
strongly coercive at $y$ on bounded sets}.
If \eqref{e:scoer} holds with one and the same constant $\mu>0$ on all neighborhoods $V$, then $\varphi$
 is said to be {\em uniformly} strongly coercive at $y$.  
\end{enumerate}
\end{definition}

Clearly strong coercivity implies pointwise quadratic supportability, but the reverse implication 
does not hold, as the next example shows.  
\begin{example} The Huber function (see \eqref{e:Huber} in Section \ref{s:numerics}) common in
robust regression is an example of a smooth convex function that does not satisfy \eqref{e:psc} at all points 
(namely where it is linear), but does satisfy this inequality at its minimum.  
A slight modification of the Huber function is 
\begin{equation}\label{e:Huberl}
\qquad  \phi_{\alpha}(t) =
\begin{cases} \frac{ (t+\epsilon)^2 -\epsilon^2}{2 \alpha} & \text{if $0\leq t \leq \alpha-\epsilon$} \\
                       \frac{ (t-\epsilon)^2-\epsilon^2 }{2 \alpha} & \text{if $-\alpha+\epsilon\leq t \leq 0$} \\
                       |t| +\left(\epsilon -\frac{\epsilon^2+\alpha^2}{2\alpha}\right) & \text{if $|t| > \alpha-\epsilon$}.
\end{cases}  
\end{equation}
This function is convex and pointwise quadratically supportable on 
bounded sets at $t=0$ where $\partial \phi_{\alpha}(0)=[-\epsilon/\alpha, \epsilon/\alpha]$.  This 
function is not, however, strongly convex due to the linear portion for $|t| > \alpha-\epsilon$. 
 \end{example}

 For convex functions, pointwise quadratic supportability is {\em equivalent} to pointwise quadratic supportability 
on bounded sets, as the next proposition establishes. 
\begin{proposition}\label{t:plc2pcbs} Let $\varphi : \Rn \to  \R\cup \{+\infty\}$ be proper
extended-valued and convex on $\Rn$.
The following are equivalent. 
\begin{enumerate}[(i)]
   \item $\varphi$ is pointwise quadratically supportable at $y$;
\item $\varphi$ is pointwise quadratically supportable at $y$ on bounded sets.
\end{enumerate}
 \end{proposition}
\begin{proof}
That pointwise quadratic supportability at $y$ on bounded sets implies pointwise quadratic supportability
at $y$ is clear.  For the converse implication, 
fix $R>0$ and choose any  $z\in R\mathbb{B}(y)$.  
Let $V$ be the neighborhood of $y$ on which \eqref{e:psc} holds with constant $\mu$.
If $z\in V$, then inequality \eqref{e:psc} is trivially satisfied.  
Suppose, then, that $z\notin V$. 
Since $V$ is a neighborhood of $y$, there exists a $\delta>0$ and a point $x\in \delta
\mathbb{B}(y)\subset V$ such that $x=\tau z + (1-\tau)y$ for $\tau\in (0,1)$ and $\|x-y\|=\delta$.  
Then 
\[
 z-y = \frac{1}{1-\tau}(z-x) = \frac{1}{\tau}(x-y)
\]
and 
\[
 \tau = \frac{\delta}{\|z-y\|}\geq \frac{\delta}{R}.
\]
By the characterization of the convex subdifferential we have
\[
   \varphi(z)\geq \varphi(y)+ \bform{v, z-y} ~\forall ~v\in\partial \varphi(y).
\]
Adding $\varphi(x)-\varphi(x)$ to this and using \eqref{e:psc} yields
\begin{eqnarray*}
(\forall ~v\in\partial \varphi(y))\qquad \varphi(z)&\geq& \varphi(y)+ \left(\varphi(y)+\bform{v, x-y}+\frac{\mu}{2}\|x-y\|^2\right)-\varphi(x)+ \bform{v, z-y}\\
&=& 2\varphi(y)+(1+\tau)\bform{v, z-y}+\frac{\mu \tau^2}{2}\|z-y\|^2-\varphi(x)\\
&\geq& 2\varphi(y)+(1+\tau)\bform{v, z-y}+\frac{\mu \tau^2}{2}\|z-y\|^2-\left(\tau \varphi(z)+(1-\tau)\varphi(y)\right)\\
&=& (1+\tau)\varphi(y)-\tau \varphi(z) + (1+\tau)\bform{v, z-y}+\frac{\mu \tau^2}{2}\|z-y\|^2.
\end{eqnarray*}
Rearranging the terms and simplifying yields 
\[
(\forall ~v\in\partial \varphi(y))\qquad   \varphi(z)\geq \varphi(y) + \bform{v, z-y}+\frac{\mu \tau^2}{2(1+\tau)}\|z-y\|^2\geq 
\varphi(y) + \bform{v, z-y}+\frac{\mu \delta^2}{4R^2}\|z-y\|^2.
\]
Since $R$ and $z$ are arbitrary, this completes the proof. 
\end{proof}

It is worthwhile contrasting the above property to the assumption of {\em strong convexity} 
that is common in the literature.  Analogous to Definition \ref{d:psc}\eqref{d:pscoer}, we 
generalize this notion to {\em pointwise strong convexity}, which is new. 
\begin{definition}[(pointwise) strongly convex functions]\label{d:(p)sc}  $~$
\begin{enumerate}[(i)]
\item\label{d:pscvx} A function  $\varphi : \Rn \to  \R\cup \{+\infty\}$ 
is said to be {\em pointwise strongly convex at $y$} if there exists a convex neighborhood $V$ of 
$y$ and a constant $\mu>0$
such that,  
 \begin{equation}\label{e:pscvx}
(\forall \tau\in (0,1))\quad \varphi\left(\tau x+(1-\tau)y\right)\leq \tau \varphi(x)+(1-\tau) \varphi(y)-\frac12\mu \tau(1-\tau)\|x-y\|^2,  \quad \forall  x\in V. 
\end{equation} 
If for each bounded convex neighborhood $V$ there exists a constant $\mu>0$ such that \eqref{e:pscvx} 
holds, then $\varphi$ is said to be {\em pointwise strongly convex at $y$ on bounded sets}.  
If there exists a single constant $\mu>0$ such that \eqref{e:pscvx} 
holds on all convex neighborhoods $V$, then $\varphi$ is said to be {\em uniformly} pointwise strongly convex
at $y$.  
\item\label{d:scvx} A function  $\varphi : \Rn \to  \R\cup \{+\infty\}$ 
is said to be {\em strongly convex at $y$} if there exists a convex neighborhood $V$ of 
$y$ and a constant $\mu>0$
such that,  
 \begin{equation}\label{e:scvx}
(\forall \tau\in (0,1))\quad \varphi\left(\tau x+(1-\tau)z\right)\leq \tau \varphi(x)+(1-\tau) \varphi(z)-\frac12\mu \tau(1-\tau)\|x-z\|^2,~ \forall  x,z\in V. 
\end{equation} 
If for each bounded convex neighborhood $V$ there exists a constant $\mu>0$ such that \eqref{e:scvx} 
holds, then $\varphi$ is said to be {\em strongly convex at $y$ on bounded sets}.  
If there exists a single constant $\mu>0$ such that \eqref{e:scvx} 
holds on all convex neighborhoods $V$, then $\varphi$ is said to be {\em uniformly} 
strongly convex
at $y$.  
\end{enumerate}
\end{definition}
Again, it is clear from the definition that strong convexity implies pointwise 
strong convexity, but the converse need not, in general, be true.  Indeed, 
it is well known in the smooth case ($\varphi$ continuously differentiable)
that strong convexity defined by \eqref{e:scvx} is {\em equivalent}
to strong coercivity defined by \eqref{e:scoer}, with the same constants on 
the same neighborhoods.  The situation is very different for the 
pointwise definition, even for smooth functions as the next example demonstrates.
\begin{example}[pointwise quadratic supportability does not imply convexity]
 The function $\varphi(y)\equiv 1-e^{-y^2}$ is 
 pointwise quadratically supportable on bounded sets at $y=0$, but it is neither
convex nor pointwise strongly convex on bounded sets.  For that matter, it is not even coercive. 
\end{example}

Another notion common in the literature is strong monotonicity.  Again, for 
a smooth convex function, strong convexity is equivalent to strong monotonicity 
of its gradient.  The next result shows that pointwise strong convexity and pointwise 
strong monotonicity for differentiable functions imply pointwise quadratic supportability.  Together with 
the above example, this shows that pointwise strong 
coercivity is the weakest of the three notions.  
\begin{proposition}[pointwise quadratically supportable/convex/monotone differentiable functions]\label{t:psc smooth}
Let $\varphi : \Rn \rightarrow \R\cup\{+\infty\}$ be continuously differentiable at $y$.  
\begin{enumerate}[(i)] 
 \item\label{t:psc smooth i} If $\varphi$ is pointwise strongly convex at $y$ with constant $\mu$ on 
the convex neighborhood $V$ (that is, $\varphi$ satisfies \eqref{e:pscvx}), then $\varphi$ is pointwise quadratically supportable at 
$y$ with constant $\mu$ on $V$ (that is, $\varphi$ satisfies 
\eqref{e:psc}).
 \item\label{t:psc smooth iii} If there exists a constant $\mu$ and a convex neighborhood $V$ of $y$ such that 
 \begin{equation}\label{e:psm smooth}
\bform{\nabla \varphi(x)-\nabla \varphi(y), x-y}\geq \mu\|x-y\|^2 \quad \forall  x\in V,
\end{equation} 
then $\varphi$ satisfies \eqref{e:psc} with the same constant 
on some neighborhood $V'\subset V$.  
\end{enumerate}
\end{proposition}
\begin{proof}  
To see \eqref{t:psc smooth i}, rearrange the 
inequality \eqref{e:pscvx}, divide through by $\tau$ and take the limit as $\tau\to 0$.  

To see \eqref{t:psc smooth iii}  choose any $x\in V$. 
By the Mean Value Theorem  there is a $\lambda\in (0,1)$ such that, for $z=\lambda x+ (1-\lambda)y$ (which is in $V$
since this is convex),
\begin{eqnarray*}
   \varphi(x)-\varphi(y) &=& \bform{\nabla \varphi(z), x-y}\\
  &=& \bform{\nabla \varphi(y), x-y}+\bform{\nabla \varphi(z)-\nabla \varphi(y), x-y}\\
  &=& \bform{\nabla \varphi(y), x-y}+\frac{1}{\lambda}\bform{\nabla \varphi(z)-\nabla \varphi(y), z-y}\\
&\geq& \bform{\nabla \varphi(y), x-y}+\frac{\mu}{\lambda}\|z-y\|^2\\
&\geq& \bform{\nabla \varphi(y), x-y}+\lambda\mu\|x-y\|^2.
\end{eqnarray*}
By continuity of $\nabla \varphi$, $\lambda\to 1$ as $x\to y$, hence,
for all $x$ close enough to $y$, that is in some neighborhood $V'\subset V$,
$\lambda>1/2$  and 
\[
   \varphi(x) \geq \varphi(y)+ \bform{\nabla \varphi(y), x-y}+\frac{\mu}{2}\|x-y\|^2
\]
as claimed.  
\end{proof}

We leave a complete development of pointwise coercivity and related objects to future research.  
For our concrete application, namely linear image denoising and deconvolution, we will assume 
the following throughout.
\begin{center}
\fbox{%
        \addtolength{\linewidth}{0\fboxsep}%
        \addtolength{\linewidth}{0\fboxrule}%
        \begin{minipage}{\linewidth}%
        \begin{assumption}\label{hyp:B}$~$
The function $f : \Rn \rarr \R$ of Problem \ref{ch:main} is 
pointwise quadratically supportable at each $\hx\in X^*$, that is,  
there exists a $\mu>0$ 
and a convex neighborhood $V$ of $\hx$ such that 
 \begin{equation}\label{muf}
f(x)\geq f(\hx)+\left\langle \nabla f(\hx), x-\hx\right\rangle + \tfrac12\mu\|x-\hx\|^2 \quad \forall  x\in V. 
\end{equation} 
\end{assumption}
 \end{minipage}%
}
\end{center}

Note that condition \eqref{t:psc smooth iii} of Proposition \ref{t:psc smooth} puts a lower bound 
on the constant of pointwise Lipschitz continuity for $\nabla f$, namely $\mu$.  
Also, as a consequence of Proposition \ref{t:plc2pcbs} any convex function
satisfying  Assumption \ref{hyp:B} is pointwise quadratically supportable  at all points in the solution set $X^*$ 
on all bounded convex neighborhoods of these points.  This property, together with Assumption
\ref{hyp:A1} will yield global linear convergence for our proposed algorithm and, as a corollary, 
uniqueness of saddlepoints.  
  
\subsection{The Algorithm} 
The algorithm we revisit is the primal-dual (PAPC) algorithm proposed in
\cite{drori2015simple} for solving \eqref{p:sp}. It consists of a predictor-corrector gradient step for handling the smooth
part of $K$ and a proximal step for handling the nonsmooth part. 
\bigskip

\nr \fbox{\parbox{6.5in}
{
{\bf Proximal Alternating Predictor-Corrector (PAPC) for solving \eqref{p:sp}}\\
{\bf Initialization:} Let $(x^0, y^0)\in\Rn\times\Rm$, and choose the parameters $\tau$ and 
$\sigma$ to satisfy 
\begin{subequations}\label{papc}
\begin{equation}\label{papc:parameters}
\tau \in \left(0, \sfrac{L_f}\right), \quad 0 <\tau\sigma \leq \tfrac{1}{\|\A^T\A\|}.
\end{equation}
\medskip
{\bf Main Iteration:} for $k=1,2,\dots$ update $\xk,\yk$ as follows:
\begin{flalign}
&\pk = \xkp - \tau (\nabla f(\xkp) + \A\ykp);\label{papc:minx1} \\
&\text{for $i=1, \ldots, p$,}\nonumber \\
&\qquad \yk_i = \argmax_{y_i\in\Rmi}\{ \bform{A_i^T\pk, y_i} - g^*_i(y_i) -
(1/2\sigma)\normsq{y_i - \ykp_i}\} \equiv \prox_{\sigma}^{g^*_i} (\ykp_i + \sigma A_i^T\pk);
\label{papc:miny} \\
&\xk = \xkp - \tau (\nabla f(\xkp) + \A\yk).\label{papc:minx2} 
\end{flalign}
\end{subequations}
}}

\bigskip

At each iteration the algorithm utilizes
one gradient and full proximal map evaluation on the given nonsmooth function, assumed to be easy to
compute.  For our purposes, we suppose these can be evaluated exactly, though with finite
precision arithmetic this is clearly not realistic.  A version of this theory that takes inexactness 
into account is left to future research.  

The step in the dual space 
\eqref{papc:miny} can be written in a short form using the
following notation. 
Since $g(y) := g(y_1, \ldots, y_p) = \sumi[p]g_i(y_i)$, then the convex conjugate of a separable sum
of functions gives $g^*(y) := \sumi[p]g^*_i(y_i)$ and the definition of the matrix $S =
\inv{\sigma}I_m$ we immediately get that for any point $\zeta_i\in \Rmi, \; i=1, \ldots, m$, 
$$ \prox_S^{g^*} (\zeta) = (\prox_{\sigma}^{g^*_1}(\zeta_1), \prox_{\sigma}^{g^*_2}(\zeta_2),
\ldots, \prox_{\sigma}^{g^*_p}(\zeta_p)). $$
Thus the step \eqref{papc:miny} can be written in vector notation by $\yk = \prox_S^g (\ykp +
\sigma\A^T\pk)$.  It is possible to use different proximal step constants $\sigma_i$, 
$i = 1\ldots,p$, see details in \cite{drori2015simple}. The choice $\sigma_i=\sigma$ for 
$i=1, \ldots, p$ is purely for simplicity of exposition.

\medskip

In \cite{drori2015simple} the parameter choices for obtaining the
iteration complexity with respect to the ergodic sequence was given and the convergence of 
the sequence $\{(x^k,y^k)\}_{k\in\N}$ to a saddle point solution was established. As we will use 
this later, we quote the main result here.

\begin{proposition}\cite[Corollary 3.2]{drori2015simple}\label{t:convergence}
Let $\{(\pk , \yk , \xk)\}_{k\in\N}$ be the sequence generated by
the PAPC algorithm.  If Assumptions \ref{hyp:A1}(i)(ii)(iv) are satisfied then
the sequence $\{(\xk , \yk) \}_{k\in\N}$
converges to a saddle point $(\bar{x}, \bar{y})$ of $K(\cdot,\cdot)$.
\end{proposition}

The next intermediate result establishes pointwise quadratic supportability (Definition \ref{d:psc}) on bounded sets 
at all saddle points under Assumptions \ref{hyp:A1} and \ref{hyp:B}.
\begin{proposition}\label{t:muk}
Let $\{(\pk , \yk , \xk)\}_{k\in\N}$ be the sequence generated by
the PAPC algorithm.  If Assumptions \ref{hyp:A1} and \ref{hyp:B} are satisfied, then for
any primal solution $\hx$ to the saddle point problem \eqref{p:sp},  
there exists a $\mu>0$ such that 
 \begin{equation}\label{mufk}
f(\xk)\geq f(\hx)+\left\langle \nabla f(\hx), \xk-\hx\right\rangle +
\tfrac12\mu\|\xk-\hx\|^2 \quad \forall  k.
\end{equation}  
\end{proposition}
\begin{proof}
Let $\hx$ be a primal solution to the saddle point problem.
By Proposition \ref{t:convergence} the sequence $\{\xk\}_{k\in\N}$ is bounded and 
indeed converges to a primal solution $\tilde{x}$, not necessarily the same point as $\hx$. 
A function $f$ satisfying Assumption \ref{hyp:B} is pointwise quadratically supportable  
at any primal saddle point solution to \eqref{p:sp}, 
and hence by 
Proposition \ref{t:plc2pcbs}, $f$ is pointwise quadratically supportable at $\hx$
on bounded convex sets.  In other words, there exists a $\mu>0$ and 
a ball $R\mathbb{B}(\hx)$ containing the sequence 
$\{\xk\}_{k\in\N}$ on its interior such that $f$ satisfies \eqref{mufk}, as claimed.
\end{proof}

The constant $\mu$ in Proposition \ref{t:muk} depends on the choice of $(x^0, y^0)$ and so depends implicitly
on the distance of the initial guess to the point in the set of saddle point solutions.  

\section{Convergence of the iterates}
\label{s:convergence}
\bigskip

In this section we show that, under Assumptions \ref{hyp:A1} and \ref{hyp:B}, 
the \emph{primal} sequence generated by the PAPC algorithm converges R-linearly to a saddle point. 
We recall the following simple observations (see \cite{drori2015simple}).
\begin{proposition}\label{obs} 
Let $(\hx,\hy)$ be a saddle point of $K$, and let $\{(\pk , \yk , \xk)\}_{k\in\N}$ be the sequence generated by
PAPC. Then, for all $k\geq 1$, 
\begin{equation}\label{obs-eq}
\nabla f(\xkp) - \nabla f(\hx)+ \A(\yk - \hy) + \sfrac{\tau}(\xk-\xkp) = 0.
\end{equation}
\end{proposition}
\begin{proof}
This follows immediately from the update rule \eqref{papc:minx2}  together with the optimality condition for \eqref{p:sp}
$\nabla f(\hx) + \A\hy = 0$ where $(\hx,\hy)$ is a saddle point solution to \eqref{p:sp}.
\end{proof}

The next lemma uses the following shorthand notation:
\begin{equation}\label{G-def}
G:= \inv{\sigma}I_m  - \tau \A^T\A.
\end{equation}
Note that for the choice of $\tau$ given in \eqref{papc:parameters}, $G\succeq 0$.  

\begin{lemma}\label{lemlip:1}
Let $\{(\pk, \xk, \yk)\}_{k\in\N}$ be the sequence generated by the PAPC algorithm. 
Then for every $k \in\N$ and for
any $x\in\Rn$ and $y\in\Rm$ 
\begin{equation}\begin{split}\label{dr1}
K(\xk,y) - K(x,\yk) &\leq \sfrac{2}\left(\normsq{\ykp-y}_G -
\normsq{\yk-y}_G - \normsq{\ykp-\yk}_G\right) \\ &+ \sfrac{2\tau}\left(\normsq{\xkp-x} -
\normsq{\xk-x}\right) - \frac{1}{2}\left(\sfrac{\tau} - L_f\right)\normsq{\xk-\xkp}.
\end{split}\end{equation}
\end{lemma}
\begin{proof}
\nr See Lemma \cite[Lemma 3.1]{drori2015simple}.
\end{proof}

The next lemma is key to the global linear convergence results presented in the next subsection.
\begin{lemma}\label{lemlip:2}
Let $(\hx, \hy)$ be a saddle point solution for $K(\cdot,\cdot)$ and  
let $(\pk, \xk, \yk)_{k\in\N}$ be the sequence generated by the PAPC algorithm. 
If Assumptions \ref{hyp:A1} and \ref{hyp:B} are satisfied,
then there exists a $\mu > 0$ such that for every $k \in\N$ 
\begin{equation}\begin{split}\label{dr2}
\twofrac{\mu}\normsq{\xk-\hx} &\leq \sfrac{2}\left(\normsq{\ykp-\hy}_G -
\normsq{\yk-\hy}_G - \normsq{\ykp-\yk}_G\right) \\ &+ \sfrac{2\tau}\left(\normsq{\xkp-\hx} -
\normsq{\xk-\hx}\right) - \frac{1}{2}\left(\sfrac{\tau} - L_f\right)\normsq{\xk-\xkp}.
\end{split}\end{equation}
\end{lemma}
\begin{proof}
Fix  $(\hx,\hy)$, the saddle point solution to \eqref{p:sp}.  By \eqref{dr1} of Lemma \ref{lemlip:1},
we have  
\begin{eqnarray} 
\!\!\!\!K(\xk,\hy) - K(\hx,\yk) &=& K(\xk,\hy) - K(\hx,\hy) + K(\hx,\hy) -
K(\hx,\yk)\nonumber\\
\!\!\!\!&\leq& \sfrac{2}\left(\normsq{\ykp-\hy}_G -
\normsq{\yk-\hy}_G - \normsq{\ykp-\yk}_G\right) \nonumber \\
\!\!&&+ \sfrac{2\tau}\left(\normsq{\xkp-\hx} -
\normsq{\xk-\hx}\right) - \frac{1}{2}\left(\sfrac{\tau} - L_f\right)\normsq{\xk-\xkp}.
\label{dr1b}
\end{eqnarray}
Recall that for the parameter values specified in \eqref{papc:parameters} $G\succeq0$. 
By the saddle point inequality $K(\hx,\hy) - K(\hx,\yk) \geq 0$, and by Proposition \ref{t:muk}
there is a $\mu>0$ such that 
\begin{equation*}
K(\xk,\hy) - K(\hx,\hy) = f(\xk) - f(\hx) + \bform{\xk-\hx, \A\hy} = f(\xk) - f(\hx) -
\bform{\xk-\hx, \nabla f(\hx)} \geq \twofrac{\mu}\normsq{\xk-\hx},
\end{equation*} 
where in the last inequality we have used optimality condition of problem ($\mathcal{M}$),
namely $\nabla f(\hx) + \A\hy = 0$, and the result follows.
\end{proof}

Note that the saddle point in Lemma \ref{lemlip:2} need not be the limit point of the 
sequence generated by PAPC. 

\subsection{Main Results - R-Linear Rate of the Primal Sequence}

Convergence of the primal-dual sequence becomes more transparent using 
$G$ in \eqref{G-def} to define a weighted norm on the primal-dual product space.  
\begin{equation}\label{H-def}
u = \begin{pmatrix}x\\y\end{pmatrix}, \qquad H = \begin{pmatrix} \tau^{-1}I_n & 0 \\ 0 &
G \end{pmatrix}
\end{equation}
where by the assumptions on the choice of $\tau$ given in \eqref{papc:parameters}, $G\succeq 0$.  
We can then define an associated norm using the positive-semidefinite matrix $H$, 
$\normsq{u}_H := \sfrac{\tau}\normsq{x} + \normsq{y}_G$.
The main inequality in Lemma \ref{lemlip:2} written in terms of this norm then becomes
\begin{equation}\label{mb:2}
\left(\sfrac{\tau} - L_f\right)\normsq{\xk-\xkp} + \normsq{\ykp-\yk}_G +
\mu\normsq{\xk-\hx} \leq \normsq{\ukp-\hu}_H - \normsq{\uk-\hu}_H.
\end{equation}

In order to establish the Q-linear (global) convergence of the associated
sequence $\{\uk\}_{k\in\N}$ with respect to the $H$-norm, 
we show that there exists a positive $\delta$ such that 
\begin{equation}\label{rate}
\forall k\in\N: \qquad (1 + \delta)\normsq{\uk-\hu}_H \leq \normsq{\ukp-\hu}_H.
\end{equation}
This will depend implicitly on $\mu$, the constant of pointwise strong convexity of $f$, which 
depends implicitly on the distance of the initial guess to the solution set $X^*$.  
In Theorem \ref{q-rate} we will show there exists a
positive $\delta$ such that for all $k\geq 1$,
\begin{align}
\delta\normsq{\uk-\hu}_H &\leq 
\left(\sfrac{\tau} - L_f\right)\normsq{\xk-\xkp} + \mu\normsq{\xk-\hx},
\label{mb:1}
\end{align}
Since adding \eqref{mb:1} to \eqref{mb:2}
(ignoring the term $\normsq{\ykp-\yk}_G$ which is nonnegative) yields \eqref{rate}, we obtain the
$Q$-linear rate for the sequence $\{\uk\}_{k\in\N}$ with respect to the $H$-norm.

\begin{theorem}\label{q-rate}
Let $\{(\pk, \xk, \yk)\}_{k\in\N}$ be the sequence generated by the PAPC algorithm, 
and let $(\hx,\hy)$ be any saddle point solution for $K(\cdot,\cdot)$.
If Assumptions \ref{hyp:A1} and \ref{hyp:B} are satisfied, then, for any $\alpha > 1$ and for all $k\geq
1$, the sequence $\{u^k\}_{k\in\N}$ satisfies
\begin{equation}\label{u:qrate}
\normsq{\uk - \hu}_H \leq \sfrac{1 + \delta}\normsq{\ukp - \hu}_H,
\end{equation}
where
\begin{equation}\label{delta-def}
\delta = \min \left\{\frac{(\alpha-1)\tau\sigma(1 -\tau L_f)\lambda_{min}(\A^T\A)}{\alpha},
\frac{\mu\tau\sigma\lambda_{min}(\A^T\A)}{\alpha\tau L_f^2 + \sigma \lambda_{min}(\A^T\A)}\right\}
\end{equation}
is positive and $\mu>0$ is the constant of pointwise quadratic supportability of $f$ at $\hat{x}$
depending on the distance of the initial guess to the point $(\hx,\hy)$ in the solution set $X^*$.
In particular, $\{(\xk, \yk)\}_{k\in\N}$ is $Q$-linearly convergent with respect to the $H$-norm to
a saddle-point solution.
\end{theorem}
\begin{proof}
Let $(\hx,\hy)$ be any saddle point solution for $K(\cdot,\cdot)$.
Using \eqref{obs-eq} in Proposition \ref{obs}
$$ \sfrac{\tau}(\xk-\xkp) = \nabla f(\hx) - \nabla f(\xkp) + \A(\hy - \yk), $$
and by adding and subtracting $\nabla f(\xk)$ and rearranging, we have
\begin{equation}\label{pr:0}
\sfrac{\tau}(\xk-\xkp) - \left(\nabla f(\xk) - \nabla f(\xkp)\right)= \nabla f(\hx) - \nabla
f(\xk) + \A(\hy - \yk).
\end{equation}
Applying the inequality $\normsq{a+b} \geq (1 - \alpha)\normsq{a} + (1 -
\sfrac{\alpha})\normsq{b}\quad \forall \alpha>0$ to \eqref{pr:0}
with $a:= \nabla f(\hx) - \nabla f(\xk)$, and $b:=  \A(\hy - \yk)$, we have
\begin{eqnarray}
&&\Big\|\sfrac{\tau}\left(\xk-\xkp\right) - \left(\nabla f(\xk) - \nabla f(\xkp)\right)\Big\|^2\nonumber\\
&&\qquad\qquad\geq (1 - \alpha )\normsq{\nabla f(\hx) - \nabla f(\xk)} + \left(1 -
\sfrac{\alpha}\right)\normsq{\A(\hy - \yk)}.
\label{pr:1}
\end{eqnarray}
Next we have 
\begin{subequations}\label{pr:2}
\begin{align}
&\Big\|\sfrac{\tau}\left(\xk-\xkp\right) - \left(\nabla f(\xk) - \nabla f(\xkp)\right)\Big\|^2
\nonumber \\ &= \sfrac{\tau^2}\Big\|\xk-\xkp \Big\|^2 - \frac{2}{\tau}\bform{\xk-\xkp,\nabla f(\xk)
- \nabla f(\xkp)} + \normsq{\nabla f(\xk) - \nabla f(\xkp)} \nonumber \\
&\leq \sfrac{\tau^2}\Big\|\xk-\xkp \Big\|^2 + \left(1 - \frac{2}{\tau L_f}\right)
\normsq{\nabla f(\xk) - \nabla f(\xkp)} 
\label{pr:2a}\\
&\leq \sfrac{\tau^2}\|\xk-\xkp \|^2,
\label{pr:2b}
\end{align}
\end{subequations}
where in inequality \eqref{pr:2a} we used $\bform{\xk-\xkp,\nabla f(\xk) - \nabla
f(\xkp)} \geq \sfrac{L_f}\normsq{\nabla f(\xk) - \nabla f(\xkp)}$ (cf.
\cite[Thm. 2.1.5]{nesterov2004introductory}, and in  inequality \eqref{pr:2b} we used 
the fact that $\tau L_f < 1$.
Adding \eqref{pr:1}-\eqref{pr:2} we get 
$$ \sfrac{\tau^2}\normsq{\xk-\xkp} + (\alpha
- 1)\normsq{\nabla f(\hx) - \nabla f(\xk)} \geq \left(1 - \sfrac{\alpha}\right)\normsq{\A(\hy - \yk)}. $$
For any $\alpha > 1$ (so that $1- \sfrac{\alpha} > 0$), we can bound from below the right-hand side
of the later inequality   
by Assumption 1\eqref{hyp:A1iii} and bound from above the left-hand side by
Assumption 1\eqref{hyp:A1i} to yield
\begin{equation*}
\sfrac{\tau^2}\normsq{\xk-\xkp} +  (\alpha -1)L_f^2\normsq{\hx - \xk} \geq
\left(1 -
\sfrac{\alpha}\right) \lambda_{min}(\A^T\A)\normsq{\hy - \yk}.
\end{equation*}

Using the semi-norm notation to write 
$\normsq{\hy- \yk} = \sigma\normsq{\hy-\yk}_{\inv{\sigma}I_m}$
yields
\begin{equation*}
\sfrac{\tau^2}\normsq{\xk-\xkp} +  (\alpha -1)L_f^2\normsq{\hx - \xk} \geq
\left(1 - \sfrac{\alpha}\right) \sigma \lambda_{min}(\A^T\A)\normsq{\hy - \yk}_{\inv{\sigma}I_m}.
\end{equation*}
Subtracting from both sides of the inequality the term 
\[
\left(1 -\sfrac{\alpha}\right) \sigma \lambda_{min}(\A^T\A)\normsq{\hy-\yk}_{\tau\A^T\A}
 \]
 (which is positive since $\alpha>1$, $\tau > 0$
and $\A^T\A \in\mathbb{S}^m_{++}$)  together with the notation of
$\|\cdot\|_G$ defined in \eqref{G-def}, yields
\begin{equation*}
\sfrac{\tau^2}\normsq{\xk-\xkp} +  (\alpha -1)L_f^2\normsq{\hx - \xk} \geq
\left(1 - \sfrac{\alpha}\right) \sigma \lambda_{min}(\A^T\A)\normsq{\hy - \yk}_{G}.
\end{equation*}
After multiplication of both sides by $\frac{\alpha}{(\alpha-1)\sigma \lambda_{min}(\A^T\A)}$, we
get
\begin{equation*}
\left(\frac{\alpha}{\alpha-1}\right)\frac{1}{\sigma\tau^2
\lambda_{min}(\A^T\A)}\normsq{\xk-\xkp} + \frac{\alpha L_f^2}{\sigma \lambda_{min}(\A^T\A)}
\normsq{\hx - \xk} \geq \normsq{\hy - \yk}_{G}.
\end{equation*}
Finally, adding to both sides of the inequality $\sfrac{\tau}\normsq{\xk-\hx}$, and the
definition of the associated norm $\normsq{\hu - \uk}_{H} :=
\sfrac{\tau}\normsq{\hx-\xk} + \normsq{\hy-\yk}_G$, we have
\begin{equation}\label{pr:3}
\left(\frac{\alpha}{\alpha-1}\right)\sfrac{\sigma \tau^2
\lambda_{min}(\A^T\A)}\normsq{\xk-\xkp} + \frac{\alpha \tau L_f^2 + 
\sigma \lambda_{min}(\A^T\A)}{\tau\sigma \lambda_{min}(\A^T\A)}\normsq{\hx
- \xk} \geq \normsq{\hu - \uk}_{H}.
\end{equation}

Now, for $\mu>0$ satisfying \eqref{dr2} in Lemma \ref{lemlip:2}, 
we choose $\delta$ so that 
\begin{equation}\label{pr:4}
\sfrac{\tau} - L_f \geq \left[\left(\frac{\alpha}{\alpha-1}\right)\sfrac{\sigma \tau^2
\lambda_{min}(\A^T\A)}\right] \delta, \qquad \mu \geq \left[\frac{\alpha \tau L_f^2 + \sigma
\lambda_{min}(\A^T\A)}{\tau\sigma \lambda_{min}(\A^T\A)}\right] \delta.
\end{equation}
This establishes the choice of $\delta$ in \eqref{delta-def}.  Moreover,  $\delta$
is positive since $\tau \in (0,1/L_f)$ and $\alpha > 1$. Multiplying the inequality
\eqref{pr:3} by this $\delta$, we have
\begin{align*}
\left(\sfrac{\tau} - L_f\right)\normsq{\xk-\xkp} + \mu\normsq{\xk-\hx} \geq
\delta\normsq{\uk-\hu}_H.
\end{align*}
Adding the above inequality to \eqref{dr2} in Lemma \ref{lemlip:2} yields \eqref{u:qrate}.
Since by Proposition \ref{t:convergence} sequence $\{(\xk, \yk)\}_{k\in\N}$ converges 
to a saddle point $\bar{u}\equiv (\bar{x}, \bar{y})$ and the choice of saddle point $(\hx, \hy)$ 
in the argument above was arbitrary, we have, for $\delta$ satisfying \eqref{delta-def}, 
\begin{equation}\label{u:qrate2}
\normsq{\uk - \bar{u}}_H \leq \sfrac{1 + \delta}\normsq{\ukp - \bar{u}}_H
\end{equation} 
This completes the proof.
\end{proof}

In fact, the above theorem about convergence of the PAPC algorithm actually proves 
a much more striking fact that is independent of the algorithm.
\begin{corollary}[unique saddlepoint]\label{t:uniqueness}
If Assumptions \ref{hyp:A1} and \ref{hyp:B} are satisfied then the solution set $X^*$ is a singleton.
\end{corollary}
\begin{proof}
Inequality \eqref{u:qrate} holds for all $k$ for any point $\hu\in X^*$, thus $\uk\to \hu$ as $k\to\infty$  for all
$\hu\in X^*$, which can only happen if $X^*=\{\bar{u}\}$,  the limit of the sequence $\{\uk\}_{k\in\mathbb{N}}$.  
\end{proof}

We also immediately obtain the following corollary, which
establishes the expected R-linear iteration complexity for the primal sequence of PAPC for 
obtaining $\eps$-optimal solution.  A sublinear complexity of $O(1/\epsilon)$ 
in terms of saddle point gap function values was shown in \cite[Corollary
3.1]{drori2015simple} for the ergodic sequence.  In the 
next corollary an R-linear rate can be guaranteed under
the assumption of pointwise quadratic supportability and the full rank assumption on the linear mappings 
(Assumption \ref{hyp:A1}\eqref{hyp:A1iii}).
\begin{corollary}\label{r-corollary}
Suppose Assumptions \ref{hyp:A1} and \ref{hyp:B} are satisfied and let $\bar{u}=(\bar{x}, \bar{y})$ be 
the limit point of the sequence generated by the PAPC algorithm.  In order obtain
\begin{equation}\label{obt}
\norm{\xk-\bar{x}} \leq \eps \qquad (\textrm{resp.} \quad \norm{\yk-\bar{y}}_G \leq \eps),
\end{equation}
it suffices to compute $k$ iterations, with
\begin{equation}
k \geq \frac{2\log\left(\frac{C}{L_f\eps}\right)}{\delta} \qquad \left(\textrm{resp.} \quad k \geq
\frac{2\log\left(\frac{C}{\eps}\right)}{\delta}\right),
\end{equation}
where $ C = \norm{u^0 - \bar{u}}_H = \left(\sfrac{\tau}\normsq{x^0-\bar{x}} +
\normsq{y^0-\bar{y}}_G\right)^{1/2}$, and $\delta$ is given in \eqref{delta-def}.
\end{corollary}
\begin{proof}
We have by Theorem \ref{q-rate} that $\norm{\uk - \bar{u}} \leq C\omega^{k/2}$, where 
$$ C = \norm{u^0 - \bar{u}}_H = \left(\sfrac{\tau}\normsq{x^0-\bar{x}} + \normsq{y^0-\bar{y}}_G\right)^{1/2}
\qquad \text{and} \quad \omega = \left(1 + \delta\right)^{-1}. $$ 
Hence, 
$$\norm{\xk - \bar{x}} \leq \tau C \omega^{k/2} < \frac{C}{L_f}\omega^{k/2}, \quad
\normsq{\yk - \bar{y}}_G \leq C \omega^{k/2},$$
where in the first inequality we have used $\tau < \sfrac{L_f}$. 
To get \eqref{obt}, it suffices to ensure that
$\frac{C}{L_f}\omega^{k/2}\leq \eps$, which is equivalent to (since $\omega\in (0,1)$)
$$ k \geq \frac{2\log(C/(L_f \eps))}{-\log(\omega)} = \frac{2\log(C/(L_f \eps))}{-\log\left((1 +
\delta)^{-1}\right)}  = \frac{2\log(C/(L_f \eps))}{\log\left(1 +
\delta\right)} \geq \frac{2\log(C/(L_f \eps))}{\delta}, $$
where we used $1 + \log \delta \leq \delta \quad \forall \delta > 0$. The results of the R-linear
rate for dual sequence $\|\yk - \bar{y}\|_G$ follows similarly. This completes the proof.
\end{proof} 

The results we obtain in Theorem \ref{q-rate} and Corollary \ref{r-corollary} show that the
convergence rates depend on the parameter $\delta$ which also depends on the problem's data,
the controlling parameters and a free parameter. It is desirable to minimize the theoretical bound
$\sfrac{1+\delta}$ of the $Q$-linear rate in \eqref{u:qrate} or alternatively to maximize $\delta$.
Here we consider
the choice of the free parameter $\alpha > 1$ and the controlling parameters $\tau$ and $\sigma$
which minimize the $Q$-linear rate upper bound given in \eqref{u:qrate}.

\medskip

One observes that $\delta$ cannot grow without bound by any  choice of the controlling parameters
and the free parameter, thus $\delta \leq \delta_m$ where $\delta_m$ is finite.
First we set the other parameter $\sigma = \sfrac{\tau \lambda_{max}(\A^T\A)}$, and then we maximize
$\delta$ with respect to the step size parameter $\tau$.

\medskip

For convenience, 
denote the condition number of $\A$ by
$$ \kappa_{\A} := \sigma^2_{max}(\A)/\sigma^2_{min}(\A) \equiv 
\lambda_{max}(\A^T\A)/\lambda_{min}(\A^T\A).$$ We also denote the condition modulus of $f$ by 
$$ \kappa_f  \equiv L_f/\mu. $$

\nr Set $\sigma =
\frac{1}{\tau \lambda_{\max}(\A^T\A)}$ in \eqref{delta-def}, this yields the expression
\begin{equation}\label{delta3}
\delta = \min \left\{\frac{(\alpha-1)(1 -\tau L_f)\inv{\kappa}_{\A}}{\alpha},
\frac{\mu\inv{\kappa}_{\A}}{\alpha\tau L_f^2 + \frac{\inv{\kappa}_{\A}}{\tau}}\right\}.
\end{equation}
The parameter $\tau\in(0, \sfrac{L_f})$ which can maximize $\delta$ in \eqref{delta3} can be found
only from the second term inside the minimization above, and it evaluates to:
\begin{equation}\label{tau-b}
\tau_m = \frac{1}{\sqrt{\kappa_{\A}\alpha}L_f} =: \sfrac{\rho L_f},
\end{equation}
where we define the quantity $\rho = \sqrt{\kappa_\A \alpha}$.
Evaluating $\delta$ at $\tau_m$, we get
\begin{align*}\label{delta4}
\delta_m = \delta (\tau_m) &=  
\min \left\{\frac{(\alpha-1)(1 - \kappa^{-1/2}_{\A}\alpha^{-1/2})}{\alpha\kappa_\A },
\frac{1}{2\alpha^{1/2} \kappa^{1/2}_{\A}\kappa_f }\right\} \\
&= \min \left\{\frac{(\rho^2-\kappa_\A)(1 - \rho^{-1})}{\rho^2\kappa_\A },
\frac{1}{2\rho\kappa_f }\right\} 
\end{align*}
The first term is monotonically increasing and the second terms is monotonically decreasing with
respect to $\rho$, where $\rho > \sqrt{\kappa_\A}$.
To maximize $\delta$, we choose $\rho > \sqrt{\kappa_\A}$ (which corresponds to $\alpha > 1$) such
that the two terms are equal.
Some calculations show that this value can be found from the real solution of the cubic equation

$$ \rho^3 - \left(1 + \frac{\kappa_\A}{2\kappa_f}\right)\rho^2 - \kappa_\A \rho + \kappa_\A  = 0, $$ 
where $\rho > \sqrt{\kappa_\A}$.
For specific values $\kappa_\A, \kappa_f \geq 1$ the root can be found by the bisection method over
an appropriate interval, for example, when $\kappa_\A = 1, \kappa_f = 1$, the root is found in the
interval $(1,2]$ at $\rho\simeq 1.744$, then we may choose $\tau_m \simeq \sfrac{1.744 L_f} \simeq 0.57/L_f$.
In addition when $\kappa_f \gg \kappa_\A$, we benefit from taking the
gradient step size $\tau$ to be very close to $1/L_f$ as $\rho$ tends to $1$ in that case. 

Overall, the value of $\delta_m$ decreases as $\kappa_\A$ or $\kappa_f$ increase,
which affects the practical performance of the algorithm. 
We thus arrive at the very natural conclusion that the smaller the condition numbers 
of $\kappa_\A$ and $\kappa_f$, the better the performance of the algorithm.

\section{Computational Examples}\label{s:numerics}
We now illustrate the theoretical results with numerical 
experiments from image processing.
We study two convex optimization problems 
in signal processing and imaging: image denoising with $L_2$-discrepancy and 
TV regularization, followed by a study of simultaneous image deconvolution 
and denoising with statistical multisresolution side constraints.

\subsection{Total Variation - L$_2$-discrepancy and TV regularization}

We test PAPC for a discrete $L_2$-TV signal denoising problem introduced in 
\cite{rudin1992nonlinear}.
We apply the discretized
variational problem of the total variation based signal denoising model according to
\begin{equation}\label{den:main}
\min_{x\in X} \lambda \|\nabla x\|_1 + \twofrac{1}\normsq{x - b}_2, 
\end{equation}
where $X$ is a finite dimensional vector space equipped with the standard scalar product
 $$ \bform{x,x'} = \sum_{i}x_{i}x'_{i}, \qquad \forall x,x'\in X. $$
The gradient $\nabla : X\rarr Y$ is a vector in the vector space $Y=X$ and is defined by forward
finite-difference operator with Dirichlet boundary conditions,
 $b \in X$ is the noisy signal, and $\lambda > 0$ is a positive
parameter.
The minimization problem \eqref{den:main} can be written as a saddle point problem 
$$ \min_{x\in X} \max_{y\in Y} \bform{\mathcal{L}x, y}+F(x)-G(y) $$
with primal variable $x \in X$, dual variable $y \in Y$, $F(x) = \sfrac{2}\normsq{x-b}$, 
$G(y) = \delta_P(y)$, where $P$ is given by the convex set $P= \{ y : \| y \|_\infty\leq \lambda\}$ 
and $\| y \|_\infty$ denotes the discrete maximum norm and the function $\delta_P$ denotes the
indicator function of the set $P$, i.e., $$\delta_P(y)= \begin{cases} 0 & \text{if $y \in P$}, \\
\infty & \text{otherwise}, \end{cases} $$
and  $\mathcal{L} = \nabla$. The adjoint $\nabla^*$ is the negative  divergence, i.e.,
the unique linear mapping $\textrm{div} : Y \rarr X$ which satisfies $\bform{x, \nabla y}_Y =
\bform{ x, \nabla^* y}_X =  - \bform{x, \textrm{div}\;y}_X$ for all $x \in X, y \in Y$.
Finite differences are used for the discrete operator
$\nabla$ and its adjoint operator $\nabla^* = -\textrm{div}$ with the Dirichlet boundary conditions:
$$
(\partial x)_i = \begin{cases}
x_{i+1} - x_{i} & \text{if $1\leq i < n$} \\
 -x_{n} & \text{if $i=n$}
\end{cases},
\qquad 
(\partial^* x)_i = \begin{cases} x_{1} & \text{if $i=1$,} \\
 x_{i}-x_{i-1} & \text{if $1<i<n$,} \\
 x_{n} - x_{n-1} & \text {if $i = n$.}\end{cases}
$$
In each iteration of the PAPC algorithm we must evaluate the gradient of a smooth (strongly)
convex function $F(\cdot)$ as well as the resolvent of $G$ $$ (I + \sigma \partial G)^{-1}(y) =
\argmin_{y'\in Y} \left\{\sfrac{2}\normsq{y'-y} + \mathcal{I}_{\| y \|_\infty\leq \lambda}(y')\right\} =
\frac{y}{\max\{1, |y|/\lambda\}}. $$
The bound on the square of norm of the operator $\nabla $ (resp. \textrm{div}) in the one
dimensional case, $$ \|\nabla\|^2 = \|\textrm{div}\|^2 \leq \sumi[m-1](x_i - x_{i+1})^2 \leq 2
\sumi[m-1](x_i^2 + x_{i+1}^2) \leq 4 \sumi[m] x_i^2 \leq 4. $$ 

Since the term $F(\cdot)$ is strongly convex and continuously differentiable with Lipschitz
continuous gradient, and the operator in the bilinear term satisfies $\mathcal{L}^T\mathcal{L} = -
\Delta$ (negative of the Laplacian operator) is positive definite under the Dirichlet boundary 
condition (see Appendix A), we obtain the expected $R$-linear rate for the sequence 
$\{\xk\}_{k\in\N}$ based on the theory developed in Section \ref{s:convergence}.

We illustrate the linear convergence of the method for finding the solution of the
L$_2$-total-variation denoising problem in $\R^N$ with discretization $N=256$. The noisy
image is perturbed by an additive Gaussian noise with standard deviation $0.03$. The
regularization parameter $\lambda$ was set to $0.05$.
The denoising result is shown in Figure \ref{fig:test1}(c). In the numerical experiments, we monitor
the steps $\|\ukn - \uk\|_H$, where $\{u\}_{k\in\N}$ is the primal-dual PAPC 
sequence.  
The a posteriori upper bound on the distance of the $k$th iterate to the true solution is $\|\uk-
\hu\|_H \leq \frac{c}{1-c}\|\uk - \ukp\|_H$.  With an estimated convergence rate of  $c=0.96$ for the run
with $\tau=0.05$ and $c = 0.985$ for the run with $\tau=0.9$, this corresponds to an a posteriori
upper estimate of the pointwise error at the iteration $k=220$ of
$1.0*10^{-6}$, or about five digits of accuracy at each pixel, for the run with $\tau=0.05$, and
$2.1*10^{-5}$, or about 4 digits of accuracy at each pixel, for the run with $\tau=0.9$.

\begin{figure}[htp!]
\centering
\begin{subfigure}[]
{\includegraphics[width=3.0in]{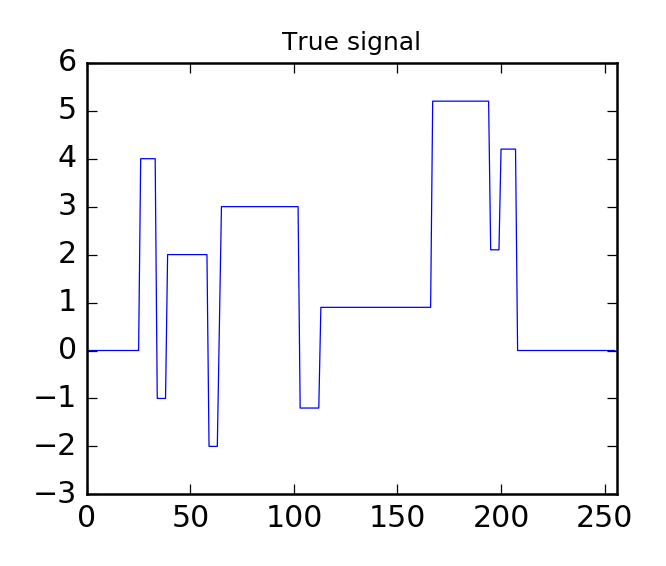}}
\end{subfigure}
\begin{subfigure}[]
{\includegraphics[width=3.0in]{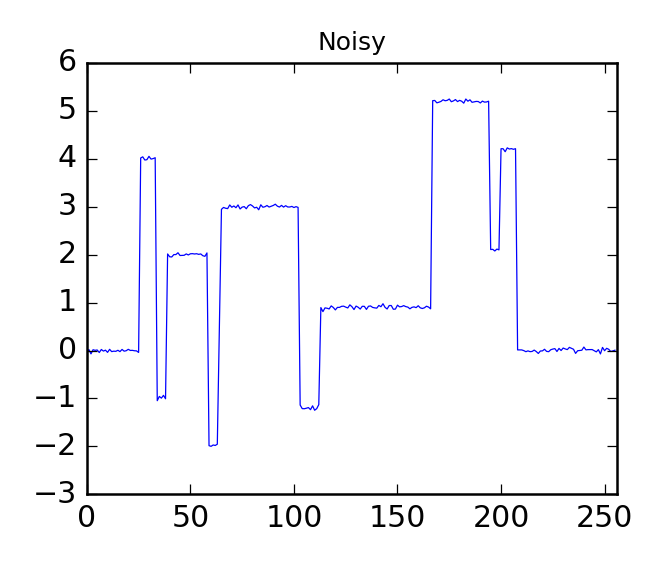}}
\end{subfigure}
\begin{subfigure}[]
{\includegraphics[width=3.0in]{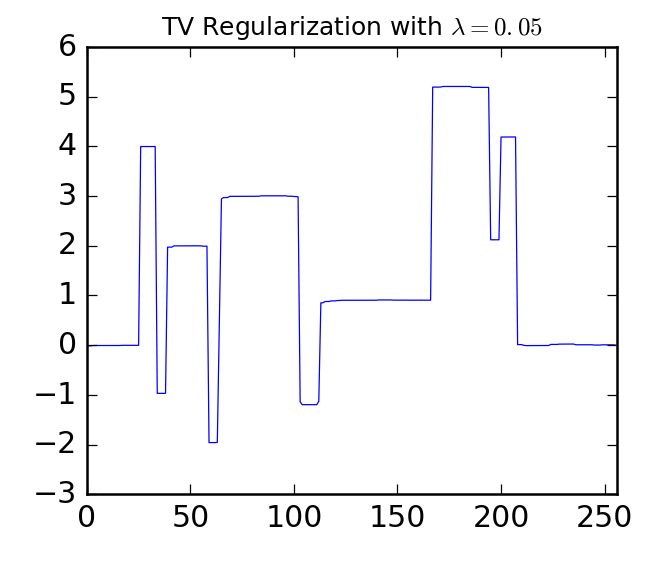}}
\end{subfigure}
\begin{subfigure}[]
{\includegraphics[width=3.1in]{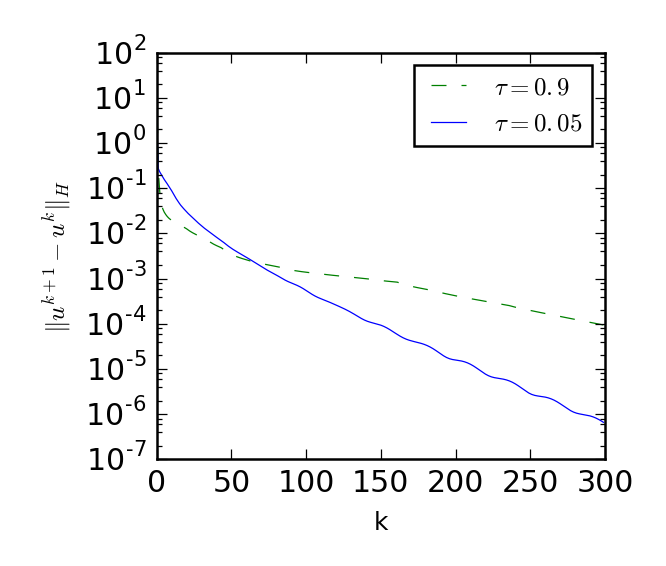}}
\end{subfigure}
\caption{{\footnotesize Performance evaluation of the PAPC algorithm for 
L$_2$-TV denoising.
(a) Original signal.
(b) Noisy signal with additive Gaussian noise, standard deviation 0.03 (c) The recovered signal
using TV regularization with $\lambda = 0.05$. (d) The a posteriori error bound for linearly convergent iteration 
 $\|\ukn - \uk\|_H$ of
the associated sequence as function of the number of iterations. The $x$-axis
represents the number of iterations, and the $y$-axis represent the distance of consecutive iterates
$\|\ukn - \uk\|_H$ in logarithmic scale. }}
\label{fig:test1}
\end{figure}

\newpage
\subsection{Statistical Multiresolution Analysis}

In this subsection we apply the spatially-adaptive method for signal and image
reconstruction that is based on Statistical MultiResolution Estimation (SMRE)
as introduced in \cite{aspelmeier2015modern,frick2012statistical,frick2013statistical}.
The variational manifestation of statistical multiresulution estimation is an optimization 
problem whose solution has a quantitative/statistical interpretation.  We are not aware 
of any other variational image processing models whose solutions have such scientific 
content.  The relevance to the theory presented in the previous sections is that, 
if one cannot estimate the distance of an iterate to a solution of the SMRE 
variational problem, then the statistical significance of the iterate generated 
by the algorithm cannot be determined -- merely knowing that the sequence 
converges, or even knowing that the objective value converges at a particular rate, 
is useless information and empty of any scientific content.

In \cite{aspelmeier2016local} a portion of the sequence (the dual sequence) of the 
alternating directions method of multipliers was shown to be locally linearly convergent, which 
yielded error estimates on solutions to iteratively regularized subproblems. We will 
give more specifics on the implementation below, but the advancement of 
the PAPC method in comparison to the ADMM method studied in \cite{aspelmeier2016local}
is two-fold:  first, we can address a wider variety of regularizing objective functions and, secondly,
the PAPC is dramatically more efficient, requiring less than $1/100$ of the time to 
achieve the kind of accuracy achieved in \cite{aspelmeier2016local}.  

The variational problem involves the construction of an estimator for an unknown {\em true} 
signal by
minimization of a convex functional $J$ over a convex set that is determined by the statistical extreme value behavior
of the residual. 
More precisely, we want to reconstruct the estimator $\bar{x}$ of the observed 
signal $b$ that is computed
as solutions of the convex optimization problem:
\begin{equation}\label{prob:smr}
\inf_{x\in X} J(x) \quad \text{s.t} \quad \max_{s\in\mathcal{S}}
\left|\sum_{\nu\in \mathcal{G}}\omega^s (Ax-b)_\nu \right| \leq q,
\end{equation}
where $J: X \rarr \R$ denotes a regularization functional, which incorporates a priori knowledge on
the unknown signal $\hx$ such as smoothness, $X$ is finite dimensional vector space,
$A : X \rarr X$ is some linear mapping, $\mathcal{S}$ denotes a system of subsets of the grid
$\mathcal{G}$ over $X$,
and $\{w^s : s\in\mathcal{S}\}$ is a set of positive weights on the grid 
which are typically normalized indicator functions on $s\in \mathcal{S}$.
The constant $q$ serves as a
regularization parameter which governs the trade-off between regularity and the fit to data of the reconstruction.

Solutions to problem
\ref{prob:smr} with SMRE constraints then have statistical content.  We are then able to obtain quantitative 
(i.e., statistical) information about  the iterates of the algorithm based on error 
bounds to solutions to \ref{prob:smr} made possible by the convergence result obtained 
in Theorem \ref{q-rate} and Corollary \ref{r-corollary}.
The problem can be rewritten as a saddle point problem as follows:
\begin{equation}
\tag{\Mcal}\label{smre:sadl}
\min_{x\in X} \max_{y\in Y}\left\{K(x,y) := J(x) + \bform{x,\mathcal{A}y} - G^*(y) \right\}
\end{equation}
where $y = (y_1, \ldots, y_{|\mathcal{\S}|})$, $y_s \in X,\; s=1, \ldots, |\mathcal{\S}|$
$\A y = \sum_{s=1}^{|\mathcal{S}|}A^Ty_i$ and $G^*$
is the conjugate function of the sum of indicator functions $ G(y) = \sum_s \delta_{C_s}(y) $ 
where each $C_s = \{ y: | \sum_{\nu \in \mathcal{G}} \omega^s (y - b) | \leq q\}$. Using the
Moreau's identity \eqref{e:Moreau}, the prox-mapping is evaluated in \eqref{papc:miny} for each
constraint by $$ \yk_s = \prox_{\sigma}^{\delta^*_{C_s}}(\ykp_s + \sigma A\pk)
= \ykp_s + \sigma A\pk - \sigma P_{C_s}\left(\frac{\ykp_s
+ \sigma A\pk}{\sigma}\right), \qquad s=1, \ldots, |\mathcal{S}|.$$ 
The proximal parameter is a function of $\tau $ and given by $\sigma= 1/ ( \tau  \|A^TA\|_2)$.
More details in \cite[Sect. 4.1]{drori2015simple}.

Our approach here is to use $J(x) = \sfrac{2}\normsq{\nabla x}$ and the smoothed TV norm.
Because the gradient $\nabla$  has a nontrivial kernel, we cannot guarantee generically that the 
objective $J$ is pointwise quadratically supportable.  We recover strong coercivity by imposing appropriate  
boundary conditions.    In the two-dimensional example, this is combined with a Huber function to 
demonstrate the full extent of the theory. 

\subsubsection{One-dimensional SMRE -- synthetic data}
In this example, we consider a one dimensional signal as shown in Figure \ref{fig:test-smr1d}(a),
and the corresponding noisy data Figure \ref{fig:test-smr1d}(b) with $n = 512$ data points.
The noisy data was generated by adding independently and identically distributed Gaussian random 
noise with standard deviation $0.02$ to each measured data point.
In this problem we consider only denoising, that is,  the operator $A$ in \eqref{prob:smr} is the
identity operator.  The qualitative objective is $J(x) = \sfrac{2}\normsq{\nabla x}$.   

The constraints are applied on all windows consisting of all 
intervals of length between $1$ and $L$ pixels.
The weights $\{w_j \in \Rn\}$ are scaled so that the windows are normalized, 
and $\mathcal{S}$ is the index set corresponding to all collections of successive 
pixels in $\{1, 2, . . . ,n\}$ of length from 1 to $l$, $1\leq l \leq L$,
where $L$ represents the total number of multiresolution levels used in the construction.
For a signal length $n = 512$ and $L=10$ multiresultion levels,
the number constraints are $|\mathcal{S}| = 5075$.  The constraints/windows have 
a separable block structure on nonoverlapping windows which can be exploited
for efficient implementation.   This allows us 
 to perform the parallel block updates in the dual space using $\sfrac{2}L(L+1)$ dual variables of
length $n$.
Figure \ref{fig:conf-smr1d} shows the schematic tiling of the nonoverlapping window constraints using
lengths varying up to $L=3$.

The objective 
function, $J(x) = \sfrac{2}\normsq{\nabla x}$, is  strongly convex if for every bounded 
neighborhood of a solution $x$, $q$ is
small enough such that there are no constant functions over the support of $3$ pixels around
$x_{i}$ which satisfy the constraint $|x_{i} - b_{i}| \leq q, \; 1 \leq i \leq n$.
We set the error
bound $q$ starting with $q_0$ corresponding to 3 standard deviations for the first level and for each subsequent 
multiresolution it is decreased based on
the window length by a factor $q_j = q_0 f^{l-1}$ is where $f\in(0,1)$ is a scaling factor and
$l\in\{1, \ldots, L\}$ is the scaling level. 
The last computed iterate is shown in Figure \ref{fig:test-smr1d}(c)
using one and ten multiresolution levels demonstrating the advantage of the
SMRE approach. With an estimated convergence rate of  $c=0.999$ for the run
with $\tau=0.2$ and $c = 0.9985$ for the run with $\tau=0.02$, this corresponds to an a posteriori
upper estimate of the error at iteration $k=1200$ of  $1.83*10^{-3}$, 
or two digits of accuracy at each pixel, and $1.09*10^{-4}$, or three digits of accuracy per pixel, 
respectively.

\begin{figure}
\centering
{\includegraphics[width=2.3in]{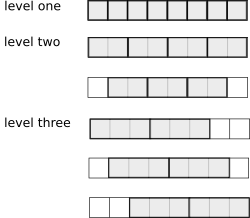}}
\caption{{\footnotesize Schematic window tiling of the multiscale constraints in the 1D case.}}
\label{fig:conf-smr1d}
\end{figure}

\begin{figure}[htp!]
\centering
\begin{subfigure}[]
{\includegraphics[width=2.9in]{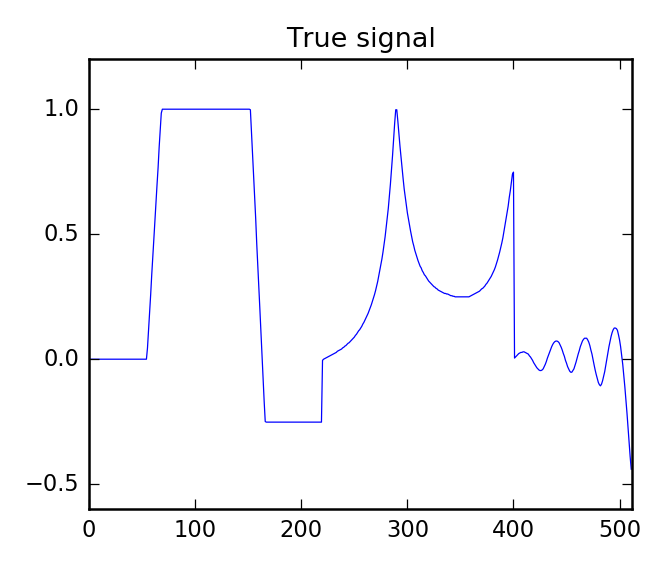}}
\end{subfigure}
\begin{subfigure}[]
{\includegraphics[width=2.9in]{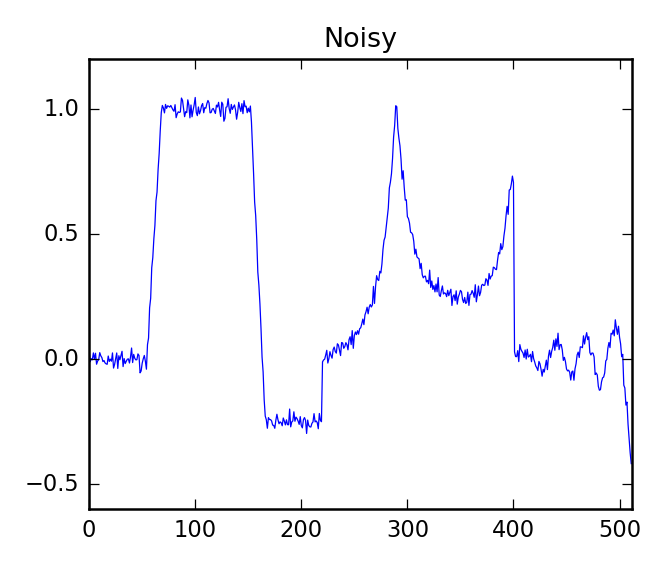}}
\end{subfigure}
\begin{subfigure}[]
{\includegraphics[width=2.9in]{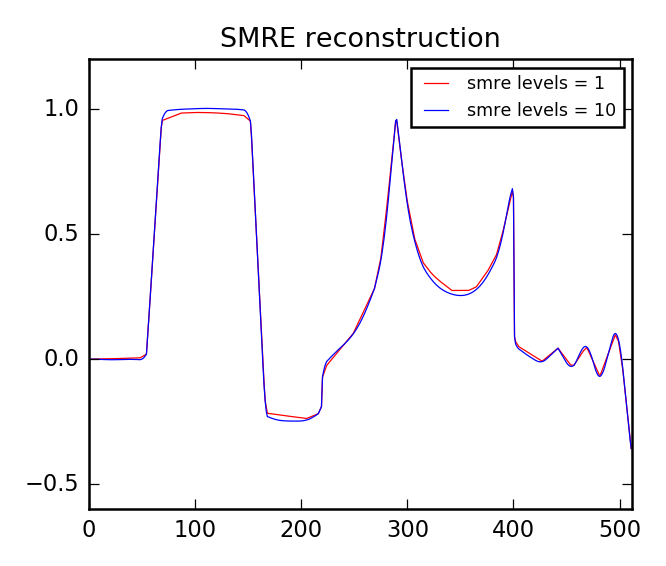}}
\end{subfigure}
\begin{subfigure}[]
{\includegraphics[width=2.9in]{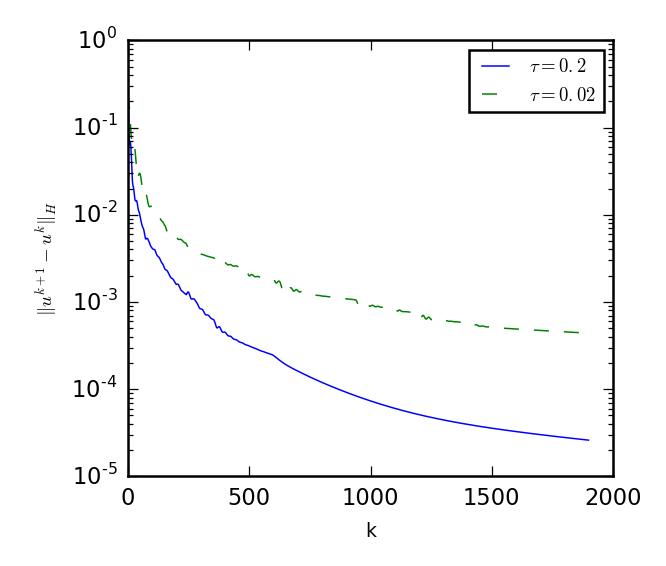}}
\end{subfigure}
\caption{{\footnotesize Results of statistical multiresolution reconstruction of
one dimensional signal with $L=1$ and $L=10$ levels using $q$ corresponding to $3$ standard deviations 
and scaling factor $f=0.93$.
(a) The original signal. (b) Noisy with additive Gaussian noise, standard deviation 0.02 (c)
The multiresolution signal reconstruction with number of levels $L=1$ and $L=10$. (d) 
The a posteriori error bound for linearly convergent iterations given by the distance of consecutive iterates
 $\|\ukn - \uk\|_H$ of the associated sequence as function of the no. of iterations. The $x$-axis
represents the number of iterations, and the $y$-axis represent the distance of consecutive iterates
$\|\ukn - \uk\|_H$ in logarithmic scale.
}}
\label{fig:test-smr1d}
\end{figure}

\pagebreak
\subsubsection{Two-dimensional  SMRE -- Image Laboratory Data}
We next consider the performance of our method in image reconstruction problems for 
Stimulated Emission Depletion (STED) microscopy experiment conducted at the 
Laser-Laboratorium G\"ottingen examining tubulin.
The SMRE model problem \eqref{prob:smr} described at length for the 1D domain can be
easily extended for an image deconvolution and denoising in the 2D case, where the linear operator
$A$ is the convolution using a point spread function.  

Here, we also consider the smooth approximation of total variation (TV) functional as the
qualitative objective. The TV functional is defined by $\|\nabla x\|_1$ where $x\in X=\R^{n\times
n}$ and is often used for image filtering and restoration. However, this function is
non-smooth where $|\nabla x| = 0$, thus in order to make the derivative-based methods possible 
we consider a smoothed approximation of the TV functional known as the so-called
Huber approximation.

The Huber loss function is defined as follows:
\begin{equation}\label{e:Huber}
 \|x\|_{1, \alpha} = \sum_{i,j} \phi_{\alpha}(x_{i,j}), \qquad  \phi_{\alpha}(t) =
\begin{cases} \frac{t^2}{2\alpha} & \text{if $|t| \leq\alpha$} \\
|t| - \twofrac{\alpha} & \text{if $|t| > \alpha$},
\end{cases}  
\end{equation}
where $\alpha > 0$ is a small parameter defining the trade-off between quadratic regularization (for
small values) and total variation regularization (for larger values). The function $\phi$ is smooth with
$\sfrac{\alpha}$-Lipschitz continuous derivative and  its derivative is given by 
\begin{equation}\label{e:DHuber}
\phi'_{\alpha}(t) = \begin{cases} \frac{t}{\alpha} & \text{if $|t| \leq\alpha$} \\
 \textrm{sgn}(t) & \text{if $|t| > \alpha$}.
 \end{cases}
\end{equation}

The objective functional
that is  used in \eqref{prob:smr} is $J(x) = \|\nabla x\|_{1,\alpha}$, where $x\in X=\Rnn$, and 
$\nabla$ is the vector in the vector space $Y = X \times X$. For
discretization of $\nabla : X\rarr Y$, we use standard finite differences with Neumann boundary
conditions, 
$$ (\nabla x)_{i,j} = \begin{pmatrix}(\nabla_1 x)_{i,j}\\ (\nabla_2 x)_{i,j}\end{pmatrix}, \qquad
1\leq i,j \leq n, $$ where 
$$(\nabla_1 x) = \begin{cases}
x_{i+1,j} - x_{i,j} & \text{if $1\leq i < n-1$} \\
 0 & \text{if $i=n-1$}
\end{cases}
\qquad (\nabla_2 x) = \begin{cases}
x_{i,j+1} - x_{i,j} & \text{if $1\leq j < n-1$} \\
 0 & \text{if $j=n-1$}.
\end{cases}
$$
The Laplacian computed at the pixel $(i,j)$ is given using the standard finite differences above as
$(\Delta x)_{i,j} = 4x_{i,j} - x_{i-1,j} - x_{i+1,j} - x_{i, j+1} - x_{i, j-1}$. The objective $J(x)$ is 
continuously differentiable with
Lipschitz gradient constant $L_J = \|\nabla\|^2/\alpha \leq 8/\alpha$. 
 This function is pointwise quadratically supportable at solutions as long as on the small bounded
neighborhood of a solution $\bar{x}$, where $|(\Delta \bar{x})_{i,j}| < \alpha$, $q$ is small enough that
there are no constant functions over the support of $3\times 3$ around $\bar{x}_{i,j}$ which satisfy the
constraints $|(A\bar{x})_{i,j} - b_{i,j}| \leq q, \; 1 \leq i,j \leq n$.  This much we must assume.

We demonstrate our results of reconstruction a close-up image shown in
Figure \ref{fig:testSTED} of dimension size $n\times n$ where $n=64$ using the statistical
multiresolution levels approach. The confidence level $q$ was set to $0.07$.
The graphs in Figure \ref{fig:resultSTED} show the a posteriori error bound results of the algorithm's
performance with quadratic objective $\sfrac{2}\|\nabla x\|^2$ and the Huber smooth approximation
$\|\nabla x\|_{1,\alpha}$, where $\alpha = 0.25$. The step size of the gradient step of both runs
was set to $\tau = 0.02$. 
The quadratic model  achieves a better rate
of convergence. 
With an estimated convergence rate of  $c=0.9993$ for the Huber objective
this corresponds to an a posteriori upper estimate of the error at iteration $k=800$ 
of $2.4*10^{-3}$.  With an estimated convergence rate of  $c=0.9962$ for the 
quadratic objective function this corresponds to an a posteriori upper estimate of the error at 
iteration $k=800$ of $1.5*10^{-3}$ -- about two digits of accuracy at each pixel.

\begin{figure}[htp!]
\centering
\begin{subfigure}[]
{\includegraphics[width=3.2in]{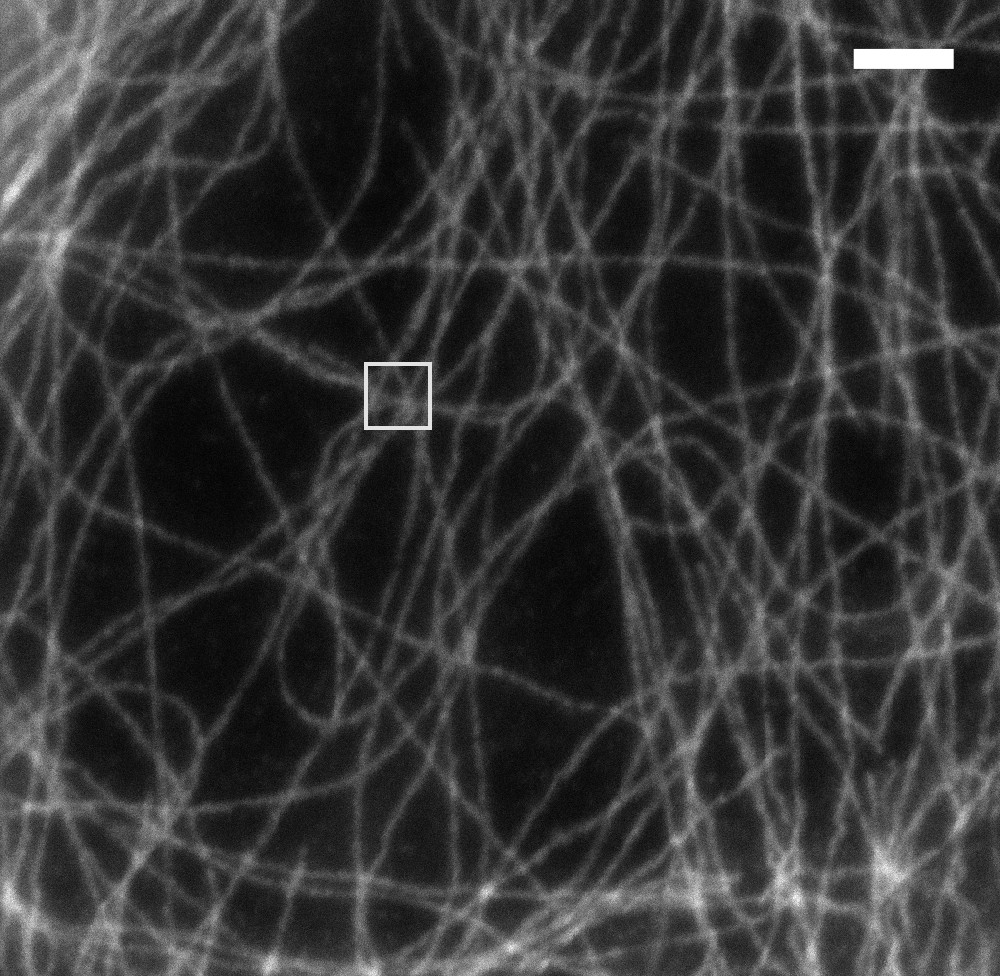}}
\end{subfigure}
\begin{subfigure}[]
{\includegraphics[width=1.7in]{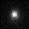}}
\end{subfigure} \\
\begin{subfigure}[]
{\includegraphics[width=3.0in]{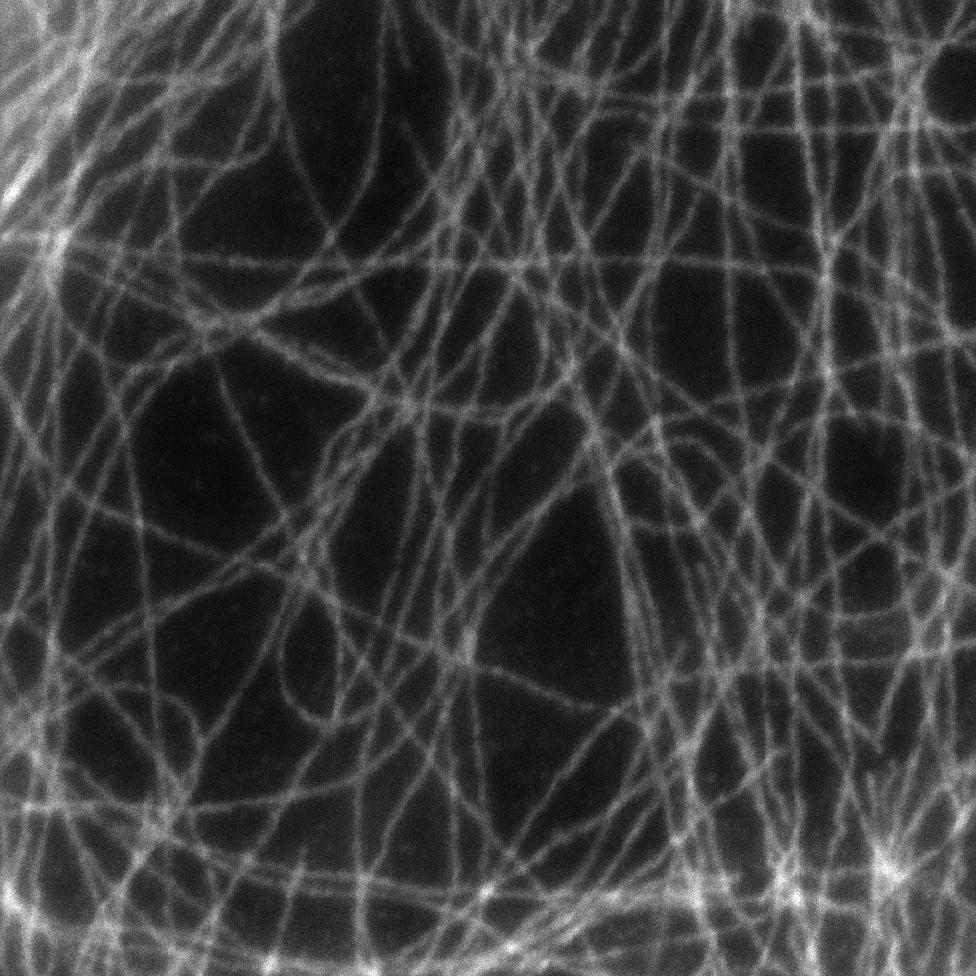}}
\end{subfigure}
\begin{subfigure}[]
{\includegraphics[width=3.0in]{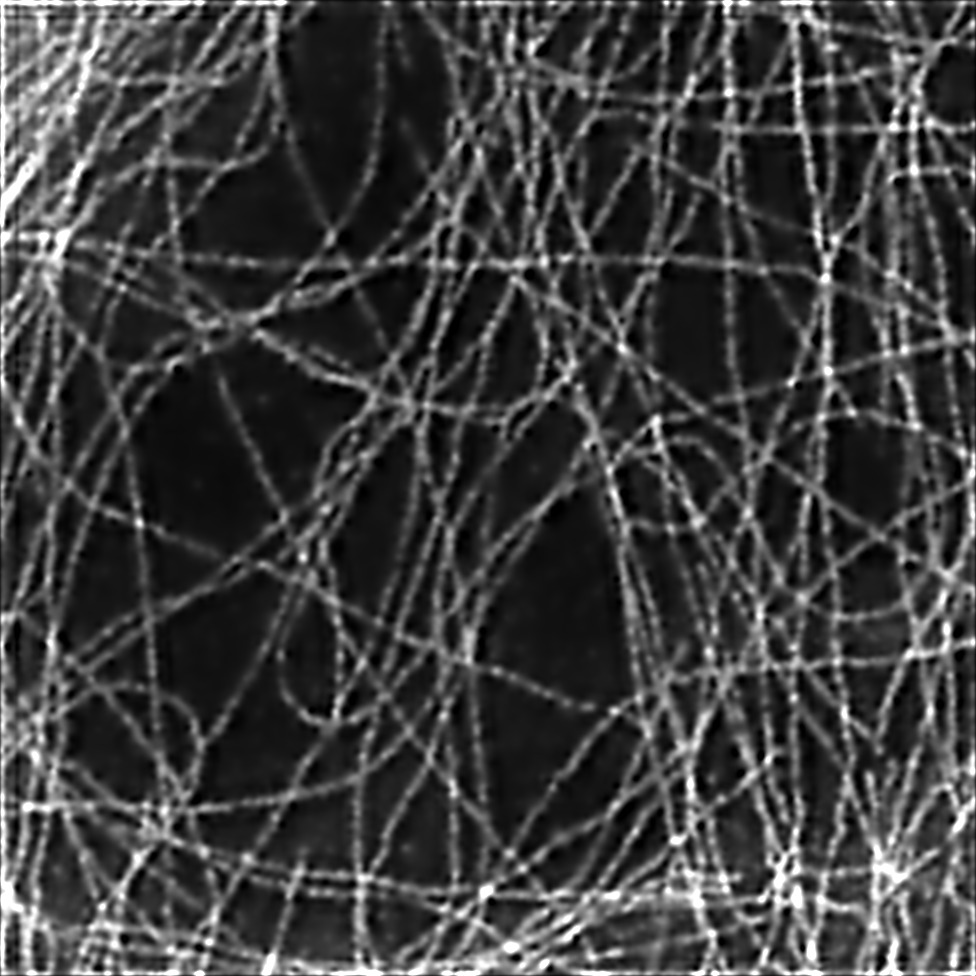}}
\end{subfigure}
\caption{{\footnotesize 
(a) Original data (STED image of Tubulin) with a close-up box to be processed, (b) PSF image,  (c)
full image deconvolution and denoising using the statistical multiresolution method with $L=3$ with
the objective functional
}}
\label{fig:testSTED}
\end{figure}

\begin{figure}[htp!]
\centering
\begin{subfigure}[]
{\includegraphics[width=2.0in]{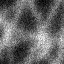}}
\label{fig1:a}
\end{subfigure}
\begin{subfigure}[]
{\includegraphics[width=2.0in]{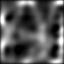}}
\end{subfigure}\\
\begin{subfigure}[]
{\includegraphics[width=2.0in]{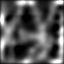}}
\end{subfigure} 
\begin{subfigure}[]
{\includegraphics[width=2.0in]{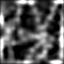}}
\end{subfigure} \\
\begin{subfigure}[]
{\includegraphics[width=3.5in]{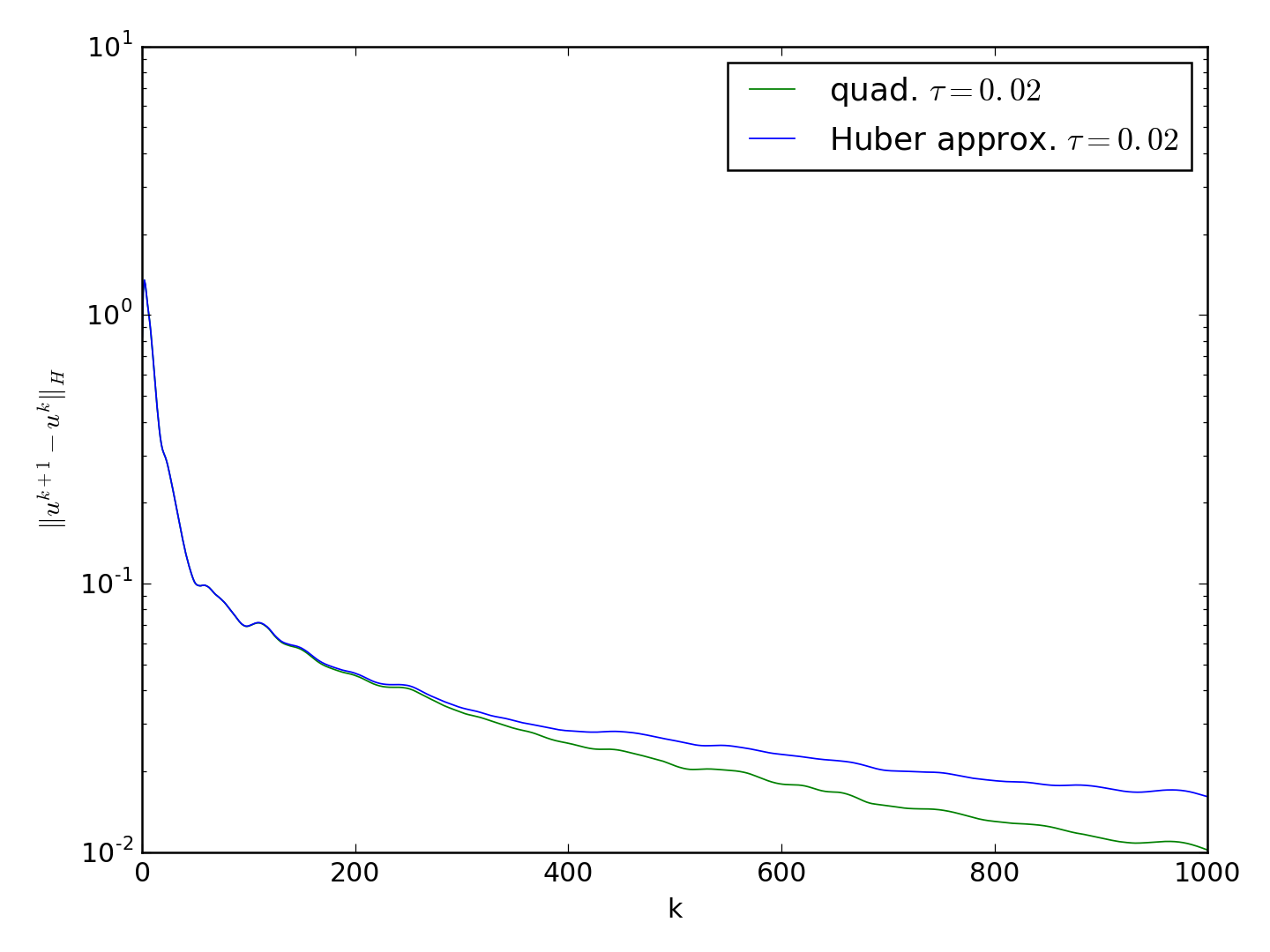}}
\end{subfigure}
\caption{{\footnotesize 
Results of image denoising and deconvolution using statistical multiresolution estimation with three levels. 
(a)-(d) depicts the evolution of the image during
reconstruction after 0, 50, 150, 1000 iterations with $\|x\|_{0.25,1}$. (e)
The primal sequence increment $\|\ukn - \uk\|_H$ of the associated sequence as function of the
number of iterations.
}}
\label{fig:resultSTED}
\end{figure}

\newpage
\pagebreak


\appendix

\section{Appendix}
It is well known that the minimum eigenvalue of the negative of the Laplacian is
positive under Dirichlet boundary condition.  
Consider the eigenvalues problem of the negative of the Laplacian operator in 1D 
\begin{equation}\label{eiglaplace}
\begin{cases}
&-\Delta x = \lambda x, \\
&x_0 = x_{n+1} = 0, \quad x\in\Rn  
\end{cases}
\end{equation}
The Laplacian operator (here the second order derivative) can be discretized using finite
differencing method, and here we consider using the standard first order discretizations 
$(-\Delta x)_j = \left(-x_{j-1} + 2x_j - x_{j+1}\right)$, $j=1, \ldots, n$.
The last expression needs to be adjusted at the boundary points based on \eqref{eiglaplace}.

Let us use the discretized solution based on the continuous case eigenvector $x_{kj} := (x_k)_j =
\sin\left(\frac{k\pi j}{n+1}\right)$ which satisfies $(x_k)_0 = (x_k)_{n+1}= 0$, and we show
that it is indeed the eigenvector $k$ computed for the component $j$ (without the normalization
constant).
We plug in to the Laplacian matrix to get: 
\begin{align*}(-\Delta x_k )_j &= 
-\sin\left(\frac{k\pi(j-1)}{n+1}\right) + 2\sin\left(\frac{k\pi j}{n+1}\right) -
\sin\left(\frac{k\pi(j-1)}{n+1}\right) \\&=  \left[2 - 2
\cos\left(\frac{k\pi}{n+1}\right)\right]\sin \left(\frac{k\pi j}{n+1}\right) \\
&= 4\sin^2\left(\frac{k \pi}{2n+2}\right) \sin \left(\frac{k\pi j}{n+1}\right).
\end{align*}
This shows that $x_{kj}$ is eigenvector with eigenvalue $4\sin^2\left(\frac{k \pi}{2(n+1)}\right)$.
The minimum eigenvalue of the nagative of the Laplacian under Dirichlet boundary condition is thus
$$\min_{\substack{x \neq 0\\ x_0=0, x_{n+1} =
0}} \frac{ -x^T\Delta x}{x^T x}  = \frac{\sumj 4\sin^2\left(\frac{\pi}{2n+2}\right)
\sin \left(\frac{\pi j}{n+1}\right)^2}{\sumj \sin
\left(\frac{\pi j}{n+1}\right)^2} = 4\sin^2\left(\frac{\pi}{2n+2}\right).$$

\section*{Acknowledgments}
We are grateful to Jennifer Schubert of the Laser-Laboratorium G\"ottingen for 
providing us with the STED measurements shown in Fig.~\ref{fig:testSTED}.
Thanks to Yura Malitsky for valuable comments during the preparation of this 
work. 


\end{document}